\documentclass[leqno,11pt]{amsart}

\usepackage[letterpaper,margin=1in]{geometry}
\usepackage{arydshln}
\usepackage[all]{xy} 
\usepackage{amsmath, amssymb, amsfonts, latexsym, mdwlist, amsthm, amscd,mathabx,braket}
\usepackage{subfig}
\usepackage{graphicx}
\usepackage{wrapfig}
\usepackage{etex}
\usepackage{tikz-cd}
\usepackage[colorlinks=true, citecolor=blue, urlcolor=blue, linkcolor=blue, pagebackref]{hyperref}

\usepackage{tikz}
\usetikzlibrary{calc,trees,positioning,arrows,chains,shapes.geometric,%
    decorations.pathreplacing,decorations.pathmorphing,shapes,%
    matrix,shapes.symbols}

\tikzset{
>=stealth',
  punktchain/.style={
    rectangle,
    rounded corners,
    draw=black, thick,
    minimum height=3em,
    text centered,
    on chain},
  line/.style={draw, thick, <-},
  eLement/.style={
    tape,
    top color=white,
    bottom color=blue!50!black!60!,
    minimum width=8em,
    draw=blue!40!black!90, very thick,
    text width=10em,
    minimum height=3.5em,
    text centered,
    on chain},
  every join/.style={->, thick,shorten >=1pt},
  decoration={brace},
  tuborg/.style={decorate},
  tubnode/.style={midway, right=2pt},
}

\usepackage{paralist}
\usepackage{enumitem} 
\usepackage{dsfont}
\setdefaultenum{(1)}{(a)}{(i)}{}
\setlist[enumerate,1]{label={\upshape(\arabic*)}}
\setlist[enumerate,2]{label={\upshape(\alph*)},ref=\theenumi\upshape(\alph*)}
\setlist[enumerate,3]{label={\upshape(\roman*)},ref=\theenumi\theenumii\upshape(\roman*)}
\usepackage[capitalize]{cleveref}
\crefname{Prop}{Proposition}{Propositions}
\crefname{Thm}{Theorem}{Theorems}
\crefname{Lem}{Lemma}{Lemmas}
\crefname{enumi}{Case}{Cases}

\def\dim{\mathop{\mathrm{dim}}\nolimits}

\def\MG13{\ensuremath{{\mathcal M}_{\Gamma_1(3)}}}
\def\tildeMG13{\ensuremath{\widetilde{\mathcal M}_{\Gamma_1(3)}}}

\newcommand\TFILTB[3]{%
\xymatrix@=1pc{
{0 = {#1}_0} \ar[rr]&&
{{#1}_1} \ar[rr]\ar[ld] &&
{{#1}_2} \ar[r]\ar[ld] &
{\cdots} \ar[r] & { {#1}_{#3-1}} \ar[rr] &&
{{#1}_{#3} = {#1}} \ar[ld]
\\
& *{{#2}_1} \ar@{.>}[ul] &&
{{#2}_2} \ar@{.>}[ul] & &&&
{{#2}_{{#3}}} \ar@{.>}[ul]
}}

\makeatletter
\newtheorem*{rep@theorem}{\rep@title}
\newcommand{\newreptheorem}[2]{%
\newenvironment{rep#1}[1]{%
 \def\rep@title{#2 \ref{##1}}%
 \begin{rep@theorem}}%
 {\end{rep@theorem}}}
\makeatother

\newtheorem{Thm}{Theorem}[section]
\newreptheorem{Thm}{Theorem}
\newtheorem{Prop}[Thm]{Proposition}

\newtheorem{Lem}[Thm]{Lemma}

\newtheorem{Cor}[Thm]{Corollary}
\newreptheorem{Cor}{Corollary}

\newreptheorem{Con}{Conjecture}

\newtheorem*{theorem*}{Theorem}
\newtheorem*{lemma*}{Lemma}
\newtheorem*{proposition*}{Proposition}
\newtheorem*{conjecture*}{Conjecture}
\newtheorem*{corollary*}{Corollary}
\newtheorem*{problem*}{Problem}

\newtheorem{Thm-int}{Theorem}

\theoremstyle{definition}
\newtheorem{Def-s}[Thm]{Definition}
\newtheorem{Def}[Thm]{Definition}
\newtheorem{Rem}[Thm]{Remark}

\newtheorem{Ex}[Thm]{Example}

\newcommand{\ignore}[1]{}
\begin{document}

\title{Hodge rigidity of Chern classes}
\author{Yuxiang Liu ${}^{1}$, Artan Sheshmani${}^{1,2,3}$ and  Shing-Tung Yau$^{2,4}$}

\address{${}^1$ Beijing Institute of Mathematical Sciences and Applications, No. 544, Hefangkou Village, Huaibei Town, Huairou District, Beijing 101408}

\address{${}^2$  Massachusetts Institute of Technology, IAiFi Institute, 77 Massachusetts Ave, 26-555. Cambridge, MA 02139, artan@mit.edu}

\address{${}^3$ National Research University Higher School of Economics, Russian Federation, Laboratory of Mirror Symmetry, NRU HSE, 6 Usacheva str.,Moscow, Russia, 119048}
\address{${}^4$ Yau Mathematical Sciences Center, Tsinghua University, Haidian District, Beijing, China}


\begin{abstract}
In this paper, we study the homogeneous components of the Chern--Schwartz--MacPherson (CSM) classes of Schubert cells. We prove that, under suitable conditions, each such component is represented by an irreducible subvariety. In particular, our result extends Huh's result \cite{Huh} by relaxing the regularity assumption on log resolutions. As a consequence, the conclusion holds for all cominuscule Schubert cells of classical type and for a large family of exceptional cases. We also obtain analogous results for certain Schubert varieties in symplectic Grassmannians and flag varieties.
    
\end{abstract}

\maketitle
\tableofcontents

\section{Introduction}
In algebraic geometry, it is natural to ask whether a given homology class can be realized by a geometrically meaningful subvariety. The Hodge conjecture addresses this question at the level of algebraic cycles, asking whether certain cohomology classes arise from algebraic subvarieties. In this paper, we consider a more restrictive representability problem for classes naturally associated with Schubert cells: whether each homogeneous component of a Chern--Schwartz--MacPherson (CSM) class is represented by an irreducible subvariety.

The CSM classes were conjectured by Grothendieck and Deligne and constructed independently by Schwartz and MacPherson. They extend the total Chern class to singular varieties, and have become an important tool in singularity theory, characteristic classes, and enumerative geometry. In particular, CSM classes can be associated with all Schubert varieties and Schubert cells. In the setting of Schubert varieties and Schubert cells, they encode subtle geometric information while retaining strong positivity and functoriality properties. This makes them a natural source of effective cycle classes whose geometric representatives are of independent interest.

Let $G$ be a semi-simple algebraic group over an algebraically closed field $k$. Fix a maximal torus $T$ and a Borel subgroup $B\supset T$. Let $P$ be a parabolic subgroup containing $B$. Let $X=\Sigma_w$, where $w\in W^P$, be a Schubert variety in $G/P$, and let $X^\circ$ denote the corresponding Schubert cell. Then
$$c_{SM}(X^\circ)=\sum\limits_{w'\leq w}\gamma_{w,w'}\sigma_{w'}\in A_*(X),$$
where $\sigma_{w'}$ are Schubert classes and $\gamma_{w,w'}\in\mathbb{Z}$ are uniquely determined coefficients. These coefficients can be computed recursively using Demazure-Lusztig type operators. Aluffi and Mihalcea obtained an explicit formula for Schubert cells in Grassmannians in \cite{AM}, and later gave an algorithm for the general homogeneous variety $G/P$ in \cite{AM2}.

Huh \cite{Huh} proved that if there exists a $B$-equivariant log-resolution $\pi:Y\rightarrow X$ such that $Y$ has finitely many $B$-orbits and, for every $y\in Y$, the isotropy Lie algebra $\mathfrak{b}_y$ contains a regular element of $\mathfrak{b}$, then each homogeneous component of $c_{SM}(X^\circ)$ equals the fundamental class of a nonempty subvariety of $X$. In particular, this holds for all Schubert varieties in Grassmannians, where Bott-Samelson varieties are isomorphic to Schubert varieties in suitable flag varieties. Our first result shows that the regularity assumption in Huh's criterion can be relaxed.
\begin{Def}
A homology class $\sigma\in A_*(X)$ is called {\em Hodge rigid} if it can be represented by an irreducible subvariety of $X$. 
\end{Def}
\begin{Thm}
Let $X$ be a Schubert variety in a homogeneous variety $G/P$, and let $X^\circ$ be the corresponding Schubert cell. Assume that there exists a $B$-equivariant log-resolution $\pi:Y\rightarrow X$ such that $Y$ has only finitely many $B$-orbits. Then each homogeneous component of $c_{SM}(X^\circ)$ is Hodge rigid.
\end{Thm}

As an application of this result, we use suitable resolutions of singularities to show that the homogeneous components of the CSM classes of certain Schubert cells are represented by irreducible subvarieties.

\subsection{Cominuscule Schubert varieties}A cominuscule homogeneous variety is one of the following:
\begin{itemize}
    \item Type $A$: Grassmannians;
    \item Type $B$: Smooth quadric hypersurfaces;
    \item Type $C$: Lagrangian Grassmannian;
    \item Type $D$: Spinor varieties;
    \item Two exceptional cases (Type $E$): the Cayley plane and the Freudenthal
variety.
\end{itemize}
We prove that every Schubert variety in a classical cominuscule homogeneous variety admits a finite log resolution. We also verify this for most Schubert varieties in the exceptional cominuscule varieties.

\begin{Thm}
    Let $X$ be a Schubert variety of a cominuscule homogeneous variety of classical types A,B,C, or D. Then each homogeneous component of $c_{SM}(X^\circ)$ can be represented by an irreducible subvariety.
\end{Thm}

\begin{Thm}
    Let $X=X(w)$ be a Schubert variety in $E_6/P_6$ with $w\leq (24315423456)$, or in $E_7/P_7$ with $w\leq (243154234567)$. Then there exists a finite log-resolution $\pi:Y\rightarrow X(w)$. In particular, each homogeneous component of $c_{SM}(X^\circ)$ is Hodge rigid.
\end{Thm}

\subsection{Symplectic Grassmannians and partial flag varieties}
Our method also applies to several non-cominuscule cases. In these situations, suitable resolutions of singularities again allow us to show that the homogeneous components of CSM classes are represented by irreducible subvarieties.
\begin{Prop}
    Let $\Sigma_{a^\alpha}(F_\bullet)$ be a Schubert variety in $F(k_1,...,k_h;n)$ (see Definition \ref{2.2}). If either $h=1$, or $k_1=1$ and $h=2$, then each homogeneous component of $c_{SM}(\Sigma_{a^\alpha}^\circ(F_\bullet))$ can be represented by an irreducible subvariety.
\end{Prop}
\begin{Prop}
    Let $\Sigma_{a;b}(F_\bullet)$ be a Schubert variety in $SG(k,n)$ (see Definition \ref{2.8}). If $s=|a|=k-1$, then each homogeneous component of $c_{SM}(\Sigma_{a;b}^\circ(F_\bullet))$ can be represented by an irreducible subvariety.
\end{Prop}

The same construction of resolutions of singularities also yields information about the singular loci of Schubert varieties in symplectic Grassmannians. In particular, we obtain the following description.

\begin{Thm}
The singular locus of $\Sigma_{a;b}(F_\bullet)$ in $SG(k,n)$ is the union of the following subvarieties (for definitions see \S \ref{singularloci}):
\begin{itemize}
    \item $\Sigma(F_{b_j})$ for all $1\leq j\leq k-s$;
    \item $\Sigma(F_{a_i})$ unless $i=s$ and $a_s=\frac{n}{2}$ and $x_{k-s}= s+1-\frac{n-2b_{k-s}}{2}$;
    \item $\Sigma(F_{b_j}^\perp)$ for all $2\leq j\leq k-s$.
\end{itemize}
\end{Thm}

\subsection*{Organization of the paper}
In \S\ref{sec-prelim}, we review the basic facts and introduce our notation for Schubert varieties and Chern--Schwartz--MacPherson (CSM) classes. In \S\ref{Hodgerigidity}, we prove that the homogeneous components of CSM classes are represented by irreducible subvarieties whenever a finite log resolution exists. In \S\ref{Cominuscule}, we apply this result to Schubert cells in cominuscule homogeneous varieties. In \S\ref{typeA}, we construct resolutions of singularities for Schubert varieties in type 
$A$ flag varieties, and show that these resolutions are finite log resolutions under suitable assumptions. Finally, in \S\ref{SG}, we study Schubert varieties and Schubert cells in symplectic Grassmannians.

\textbf{Acknowledgments.} 
A. S. would like to acknowledge grants from Beijing Institute of Mathematical Sciences and Applications (BIMSA), the Beijing NSF BJNSF-IS24005, the China National Science Foundation (NSFC) NSFC-RFIS program W2432008. He would like to thank NSF AI Institute for Artificial Intelligence and Fundamental Interactions at Massachusetts Iinstitute of Technology (MIT) which is funded by the US NSF grant under Cooperative Agreement PHY-2019786. He would also like to thank China's National Program of Overseas High Level Talent for generous support.

\section{Preliminaries}\label{sec-prelim}

\subsection{Schubert varieties}In this subsection, we introduce our notation and review some basic facts about Schubert varieties in rational homogeneous spaces. A general reference is \cite{BL}. Fix once and for all a complex vector space $V$ of dimension $n$.

\subsubsection{Grassmannians}\label{2.1.1}
The Grassmannian $G(k,V)=G(k,n)$ parametrizes $k$-dimensional vector subspaces of $V$. It is smooth and irreducible of dimension $k(n-k)$. A Schubert index $a=(a_1,\dots,a_k)$ for $G(k,n)$ is a strictly increasing sequence of integers
$$1\leq a_1<...<a_k\leq n.$$
\begin{Def}
Given a complete flag $F_\bullet: F_1\subset \cdots \subset F_n=V$ of subspaces of $V$, where the lower index denotes the dimension, and a Schubert index $a$, the Schubert variety $\Sigma_a(F_\bullet)$ is defined by
$$\Sigma_a(F_\bullet):=\{\Lambda\in G(k,n)|\dim(\Lambda\cap F_{a_i})\geq i,\ 1\leq i\leq k\}.$$
\end{Def}

\subsubsection{Partial flag varieties}
The partial flag variety $F(k_1,\dots,k_h;n)$ is defined by 
$$F(k_1,...,k_h;n):=\{(\Lambda_1,...,\Lambda_h)|\Lambda_t\subset\Lambda_{t+1},\Lambda_t\in G(k_t,V), 1\leq t\leq h\}.$$
Here we set $\Lambda_{h+1}:=V$ and $k_{h+1}:=n$. It is a smooth projective variety of dimension $\sum_{t=1}^h k_t(k_{t+1}-k_{t})$. 

The Schubert varieties in $F(k_1,\dots,k_h;n)$ can be defined as follows. A Schubert index $a^\alpha$ for $F(k_1,\dots,k_h;n)$ consists of two sequences $a=(a_1,\dots,a_{k_h})$ and $\alpha=(\alpha_1,\dots,\alpha_{k_h})$, where $a$ is a strictly increasing sequence of positive integers with $a_{k_h}\leq n$, and $\alpha$ is a sequence taking values in $\{1,\dots,h\}$ such that 
$$\#\{i|\alpha_i=t\}=k_t-k_{t-1},\ \ \ 1\leq t\leq h.$$ Here we set $k_0=0$. 

For each pair $(i,t)$ with $1\leq i\leq k_h$ and $1\leq t\leq h$, set
$$\mu_{i,t}:=\#\{i'|i'\leq i,\alpha_{i'}\leq t\}.$$

\begin{Def}\label{2.2}
Given a Schubert index $a^\alpha$ for $F(k_1,\dots,k_h;n)$ and a complete flag of subspaces $F_\bullet: F_1\subset \cdots \subset F_n=V$, the Schubert variety $\Sigma_{a^\alpha}(F_\bullet)$ is defined by
$$\Sigma_{a^\alpha}(F_\bullet):=\{(\Lambda_1,...,\Lambda_h)\in F(k_1,...,k_h;n)|\dim(F_{a_i}\cap \Lambda_t)\geq \mu_{i,t},1\leq i\leq k_h,1\leq t\leq h\}.$$
\end{Def}

\subsubsection{Orthogonal Grassmannians}
Let $q$ be a non-degenerate symmetric bilinear form on $V$. A linear subspace $W\subset V$ is called isotropic with respect to $q$ if $q(W,W)=0$. The maximal dimension of an isotropic subspace is $\left\lfloor \frac{n}{2}\right\rfloor$, where $n=\dim(V)$. For $k<\frac{n}{2}$, the orthogonal Grassmannian $OG(k,V)=OG(k,n)$ is the subvariety of $G(k,n)$ parametrizing isotropic $k$-dimensional subspaces with respect to $q$. When $n$ is even and $k=\frac{n}{2}$, the space of $k$-dimensional isotropic subspaces has two irreducible components, and we let $OG(k,2k)$ denote one of these components.

A Schubert index $(a;b)$ for $OG(k,n)$ consists of two strictly increasing sequences of integers 
$$1\leq a_1<...< a_s\leq \frac{n}{2}\ \text{ and }\ 0\leq b_1<...< b_{k-s}\leq \frac{n}{2}-1,$$
of lengths $s$ and $k-s$, respectively, such that $a_i\neq b_j+1$ for all $1\leq i\leq s$ and $1\leq j\leq k-s$. If $n=2k$, we further require that $s$ and $k$ have the same parity, namely,
$$ s \equiv k \pmod{2}.$$

Given an isotropic subspace $W$, we denote by $W^\perp$ its orthogonal complement with respect to $q$. Fix a complete flag of isotropic subspaces $F_\bullet:F_1\subset...\subset F_{\left[n/2\right]}$. When $n$ is even, the orthogonal complement of $F_{n/2-1}$ contains exactly two maximal isotropic subspaces: one is $F_{n/2}$, and the other lies in the irreducible component different from that of $F_{n/2}$. By abuse of notation, we also denote this second maximal isotropic subspace by $F_{n/2-1}^\perp$.
\begin{Def}
Given a Schubert index $(a;b)$ for $OG(k,n)$, the corresponding Schubert variety $\Sigma_{a;b}(F_\bullet)$ is defined as the Zariski closure of the locus
$$\Sigma^\circ_{a;b}(F_\bullet):=\{\Lambda\in OG(k,n)|\dim(\Lambda\cap F_{a_i})= i, \dim(\Lambda\cap F_{b_j}^\perp)= k-j+1, 1\leq i\leq s, 1\leq j\leq k-s\}.$$
\end{Def}

\begin{Rem}
For convenience, we often list the sequence $a$ from $a_1$ to $a_s$ and the sequence $b$ from $b_{k-s}$ to $b_1$, that is,
$$(a;b)=(a_1,...,a_s;b_{k-s},...,b_1),$$
so as to match the ordering from an isotropic flag
$$F_{a_1}\subset...\subset F_{a_s}\subset F_{b_{k-s}}^\perp\subset...\subset F_{b_1}^\perp.$$
\end{Rem}

\subsubsection{Orthogonal partial flag varieties}

Let $q$ be a non-degenerate symmetric bilinear form on $V$. Let $1\leq k_1<...<k_h\leq \frac{n}{2}$ be a strictly increasing sequence of positive integers. If $k_h<\frac{n}{2}$, the orthogonal partial flag variety $OF(k_1,\dots,k_h;n)$ parametrizes all $h$-step partial flags $(\Lambda_1,\dots,\Lambda_h)$ such that $\Lambda_t\subset \Lambda_{t+1}$ and each $\Lambda_t$ is an isotropic subspace of dimension $k_t$. When $k_h=\frac{n}{2}$, the space of isotropic partial flags has two irreducible components, and we let $OF(k_1,\dots,k_h;n)$ denote one of them.

A Schubert index $(a^\alpha;b^\beta)$ for $OF(k_1,...,k_h;n)$, where $n\geq 2k_h$, consists of two strictly increasing sequences of integers $$1\leq a_1^{\alpha_1}<...<a^{\alpha_s}_s\leq \frac{n}{2}$$ and $$0\leq b^{\beta_1}_1<...<b^{\beta_{k_h-s}}_{k_h-s}\leq \frac{n}{2}-1,$$ of lengths $s$ and $k_h-s$, respectively, where $0\leq s\leq k_h$, such that

\begin{enumerate}
\item $a_i\neq b_j+1$ for all $1\leq i\leq s$ and $1\leq j\leq k_h-s$, and 
\item $k_t-k_{t-1}=\#\{i|\alpha_i=t\}+\#\{j|\beta_j=t\}$ for all $1\leq t\leq h$. Here we set $k_0=0$.
\end{enumerate}
If $n=2k_h$, then we further require $s$ and $k_h$ have the same parity, i.e.
$$ s \equiv k_h \pmod{2}.$$

\begin{Def}
Given a Schubert index $(a^\alpha;b^\beta)$ for $OF(k_1,...,k_h;n)$ and a flag of isotropic subspaces $F_1\subset ...\subset F_{\left[n/2\right]}$ in $V$, set
$$\mu_{i,t}:=\#\{i'|i'\leq i,\alpha_{i'}\leq t\},$$
$$\nu_{j,t}:=\#\{i''|\alpha_{i''}\leq t\}+\#\{j'|j'\geq j,\beta_{j'}\leq t\}.$$
The Schubert variety $\Sigma_{a^\alpha;b^\beta}(F_\bullet)$ is defined as the Zariski closure of the locus
$$\Sigma^\circ_{a^\alpha;b^\beta}(F_\bullet):=\{(\Lambda_1,...,\Lambda_h)\in OF(k_1,...,k_h;n)|\dim(F_{a_i}\cap \Lambda_t)= \mu_{i,t},\dim(F^\perp_{b_j}\cap \Lambda_t)= \nu_{j,t},$$
$$\ \ \ \ \ \ \ \ \ \ \ \ \ \ \ \ \ \ \ \ \ \ \ \ \ \ \ \ \ \ \ \ \ \  \ \ \ \ \ \ \ \ \ \ \ \ \ \ \ \ \ \ \ \ \ \ 1\leq i\leq s, 1\leq j\leq k_h-s, 1\leq t\leq h\}.$$
\end{Def}

\subsubsection{Symplectic Grassmannians}
Assume that $n=\dim(V)$ is even. Let $\omega$ be a non-degenerate skew-symmetric bilinear form on $V$. A subspace $W\subset V$ is called isotropic with respect to $\omega$ if $\omega(v,w)=0$ for all $v,w\in W$. The symplectic Grassmannian $SG(k,n)$ parametrizes all $k$-dimensional isotropic subspaces of $V$ with respect to $\omega$.

A Schubert index $(a;b)$ for $SG(k,n)$ consists of two strictly increasing sequences of integers
$$1\leq a_1<...<a_s\leq \frac{n}{2}\ \text{ and }\ 0\leq b_1<...<b_{k-s}\leq \frac{n}{2}-1,$$
such that $a_i-b_j\neq 1$ for any $1\leq i\leq s$ and $1\leq j\leq k-s$. 

\begin{Rem}
For convenience, we often list the sequence $a$ from $a_1$ to $a_s$ and the sequence $b$ from $b_{k-s}$ to $b_1$, that is,
$$(a;b)=(a_1,...,a_s;b_{k-s},...,b_1),$$
so as to match the ordering from an isotropic flag
$$F_{a_1}\subset...\subset F_{a_s}\subset F_{b_{k-s}}^\perp\subset...\subset F_{b_1}^\perp.$$
\end{Rem}

\begin{Def}\label{2.8}
Given a Schubert index $(a;b)$ for $SG(k,n)$ and an isotropic flag
$$F_\bullet: F_{a_1}\subset...\subset F_{a_s}\subset F_{b_{k-s}}^\perp\subset...\subset F_{b_1}^\perp,$$
the Schubert variety $\Sigma_{a;b}(F_\bullet)$ is defined as the Zariski closure in $SG(k,n)$ of the locus
$$\Sigma^\circ_{a;b}(F_\bullet):=\{\Lambda\in SG(k,n)|\dim(\Lambda\cap F_{a_i})=i,1\leq i\leq s, $$
$$\ \ \ \ \ \ \ \ \ \ \ \ \ \ \ \ \ \ \ \ \ \ \ \ \  \ \ \ \ \ \ \ \ \ \ \ \ \ \ \ \ \ \ \ \ \ \ \ \ \ \ \ \ \dim(\Lambda\cap F_{b_j}^\perp)=k-j+1,1\leq j\leq k-s\}.$$
\end{Def}

\subsubsection{Symplectic partial flag varieties}\label{2.1.6}

Assume that $n=\dim(V)$ is even. Let $\omega$ be a non-degenerate skew-symmetric bilinear form on $V$. The symplectic partial flag variety $SF(k_1,\dots,k_h;n)$ parametrizes all $h$-step isotropic partial flags $\Lambda_1\subset...\subset \Lambda_h$, where $\dim(\Lambda_t)=k_t$ for $1\leq t\leq h$. When $h=1$, we obtain the symplectic Grassmannian $SG(k_1,n)$.

A Schubert index $(a^\alpha;b^\beta)$ for $SF(k_1,\dots,k_h;n)$, where $n\geq 2k_h$, consists of two strictly increasing sequences of integers
$$1\leq a_1^{\alpha_1}<...<a^{\alpha_s}_s\leq \frac{n}{2}$$
and 
$$0\leq b^{\beta_1}_1<...<b^{\beta_{k_h-s}}_{k_h-s}\leq \frac{n}{2}-1,$$ of lengths $s$ and $k_h-s$, respectively, where $0\leq s\leq k_h$, such that 
\begin{enumerate}
\item $a_i\neq b_j+1$ for all $1\leq i\leq s,1\leq j\leq k_h-s$, and 
\item $k_t-k_{t-1}=\#\{i|\alpha_i=t\}+\#\{j|\beta_j=t\}$ for all $1\leq t\leq h$. Here we set $k_0=0$.
\end{enumerate}

\begin{Def}
Given an isotropic flag 
$F_\bullet: F_1\subset ...\subset F_{n/2}\subset F_{n/2-1}^\perp\subset...\subset F_1^\perp\subset F_0^\perp=V$ and a Schubert index $(a^\alpha;b^\beta)$ for $SF(k_1,...,k_h;n)$, set
$$\mu_{i,t}:=\#\{i'|i'\leq i,\alpha_{i'}\leq t\},$$
$$\nu_{j,t}:=\#\{i''|\alpha_{i''}\leq t\}+\#\{j'|j'\geq j,\beta_{j'}\leq t\}.$$
The Schubert variety $\Sigma_{a^\alpha;b^\beta}(F_\bullet)$ is then defined as the Zariski closure of the locus
$$\Sigma^\circ_{a^\alpha;b^\beta}(F_\bullet):=\{(\Lambda_1,...,\Lambda_h)\in SF(k_1,...,k_h;n)|\dim(F_{a_i}\cap \Lambda_t)= \mu_{i,t},\dim(F^\perp_{b_j}\cap \Lambda_t)= \nu_{j,t},$$
$$\ \ \ \ \ \ \ \ \ \ \ \ \ \ \ \ \ \ \ \ \ \ \ \ \ \ \ \ \ \ \ \ \ \  \ \ \ \ \ \ \ \ \ \ \ \ \ \ \ \ \ \ \ \ \ \ 1\leq i\leq s, 1\leq j\leq k_h-s, 1\leq t\leq h\}.$$
\end{Def}

\subsubsection{General homogeneous varieties}\label{2.1.7}
Let $G$ be a semisimple algebraic group over an algebraically closed field $k$. Fix a maximal torus $T$ and a Borel subgroup $B\supset T$. Let $\Delta$ be the root system of $G$ relative to $T$. Let $\Delta^+$ (resp. $\Delta^-$) be the set of positive roots (resp. negative roots). The simple roots are the elements of $\Delta^+$ that cannot be expressed as a positive sum of other elements of $\Delta^+$. The set of simple roots is denoted by $\Phi$. The Weyl group $W$ is generated by the simple reflections and is isomorphic to $N(T)/T$, where $N(T)$ is the normalizer of $T$ in $G$. For $\alpha\in \Delta$, let $U_\alpha$ be the corresponding root subgroup, and let $s_\alpha$ be the corresponding reflection in $W$. 

Let $P$ be a parabolic subgroup containing $B$, and let $W_P$ be the corresponding Weyl group. The set of simple roots associated to $P$ is defined by
$$\Phi_P:=\{\alpha\in\Phi|U_{-\alpha}\not\subset P\}.$$
It is well known that the set of parabolic subgroups containing $B$ is in bijection with the power set of $\Phi$. Given a subset $I\subset \Phi$, denote by $P_I$ the parabolic subgroup such that $\Phi_{P_I}=I$. For each coset in $W/W_P$, there exists a unique element of minimal length. Let $W^P$ be the set of minimal-length representatives of $W/W_P$. For every $w\in W^P$, the $B$-stable Schubert variety $X_P(w)$ is defined as the Zariski closure of $BwP/P$ in $G/P$, endowed with its reduced induced scheme structure.

Let $w\in W^P$ and set $$\Phi_w:=\{\alpha\in\Phi_P|l(ws_\alpha)<l(w)\},$$ where $\ell$ is the length function on $W$. Note that $P_{\Phi_w}$ is the largest parabolic subgroup of $G$ such that the restriction of the canonical projection $G/P\rightarrow G/P_{\Phi_w}$ to $X_P(w)$ is birational.

\begin{Rem}
The Schubert varieties $\Sigma(F_\bullet)$ defined in \ref{2.1.1}--\ref{2.1.6} are isomorphic to $B$-stable Schubert varieties $X_P(w)$ in $G/P$, for suitable choices of $G$, $B$, and $P$. Establishing such a correspondence is standard. We will therefore blur the distinction and use these notations interchangeably in those cases.
\end{Rem}

\subsection{Chern-Schwartz-MacPherson classes}
In this subsection, we collect the necessary definitions and facts about Chern-Schwartz-MacPherson classes. For references, see \cite{AM,AM2}.

Let $X$ be a complete algebraic variety over $\mathbb{C}$. Let $F(X)$ be the group of constructible functions on $X$, generated by the characteristic functions of closed subvarieties. For $\phi\in F(X)$, consider a finite decomposition
$$\phi=\sum m_W\mathds{1}_{W}, m_W\in\mathbb{Z},$$
where the $W$ are locally closed nonsingular subvarieties of $X$. For each such $W$, there exists a desingularization $\pi:Z\to \overline{W}$ of the closure of $W$ such that $\pi^{-1}(\overline{W}\setminus W)$ is a simple normal crossings divisor with nonsingular irreducible components $D_i$.

\begin{Def}
For a locally closed nonsingular subvariety $W$ of $X$, set
    $$c_*(\mathds{1}_W):=\pi_*(\frac{c(T_Z)}{\prod(1+D_i)}\cap[Z])\in A_*(X).$$
The Chern-Schwartz-MacPherson class of $X$ is defined by
$$c_{SM}(X):=c_*(\mathds{1}_X)\in A_*(X).$$
\end{Def}

\section{Hodge rigidity of CSM classes}\label{Hodgerigidity}
In this section, we study the Hodge rigidity of CSM classes of Schubert cells. Let $G$, $B$, and $P$ be as in \S\ref{2.1.7}. Let $X=X_P(w)$ be a $B$-stable Schubert variety in $G/P$, and let $X^\circ$ be the corresponding Schubert cell, that is, the unique open dense orbit under the left action of $B$.

\begin{Def}\label{finiteresolution}
A finite log resolution of $X$ is a proper $B$-equivariant map $\pi:Y\rightarrow X$ such that
\begin{enumerate}
\item $Y$ is smooth and has finitely many $B$-orbits;
\item the induced map $\pi^{-1}(X^\circ)\rightarrow X^\circ$ is an isomorphism;
\item the complement of $\pi^{-1}(X^\circ)$ in $Y$ is a simple normal crossings divisor.
\end{enumerate}
\end{Def}

Assume that $X$ admits a finite log resolution $\pi:Y\rightarrow X$. Set $D:=Y\backslash\pi^{-1}(X^\circ)$. Define $\mathcal{L}_{Y,D,y}:\mathfrak{b}\rightarrow T_{Y,y}(-\log D)$ by sending $\xi\in\mathfrak{b}$ to the value at $y$ of the associated logarithmic vector field. Since $D$ is invariant under the action of $B$, the map $\mathcal{L}{Y,D,y}$ is well defined. Define
$$\mathfrak{X}:=\{(y,\bar{\xi})|\mathcal{L}_{Y,D,y}(\xi)=0\}\subset Y\times\mathbb{P}(\mathfrak{b}).$$
Let $\pi_{\mathfrak{X},1}:\mathfrak{X}\rightarrow Y$ and $\pi_{\mathfrak{X},2}:\mathfrak{X}\rightarrow\mathbb{P}(\mathfrak{b})$ be the two projections. Let $\Lambda$ be a general $(k+1)$-dimensional subspace of $\mathfrak{b}$, and let $\mathfrak{D}_k(\Lambda)$ be the degeneracy locus defined by $$\mathfrak{D}_{k}(\Lambda):=\pi_{\mathfrak{X},1}\circ\pi_{\mathfrak{X},2}^{-1}(\mathbb{P}(\Lambda)).$$

For each $B$-orbit $S$ in $Y$, define
$$\mathfrak{X}_S:=\pi_{\mathfrak{X},1}^{-1}(S)$$
and
$$\mathfrak{Y}_S:=\{(y,\bar{\xi})|y\in S, \xi\in\mathfrak{b}_y\}\subset Y\times\mathbb{P}(\mathfrak{b}).$$
Recall the following result:

\begin{Thm}\cite{Huh}\label{summary of Huh}
Assume that $\pi:Y\rightarrow X$ is a finite log-resolution. For a sufficiently general subspace $\Lambda\subset\mathfrak{b}$, the following hold:
\begin{itemize}
\item The $k$-dimensional component of the Chern-Schwartz-MacPherson class of $Y^\circ$ is represented by
    $$c_{SM}(Y^\circ)_k=\sum m_i[\mathcal{Z}_i]\in A_k(Y),$$
where $\mathcal{Z}_i$ are the irreducible components of $\mathfrak{D}_k(\Lambda).$
\item There is exactly one irreducible component of $\mathfrak{D}_k(\Lambda)$ that is generically supported on $Y^\circ$. Any other irreducible component is supported on a $B$-orbit $S\subset D$ such that $\mathfrak{X}_S=\mathfrak{Y}_S$ and $\text{rank}(B)>\text{rank}(\text{Stab}_B(y))$, $y\in S$.
\end{itemize}
\end{Thm}

We prove that if $X$ admits a finite log-resolution, then each homogeneous component of $c_{SM}(X^\circ)$ is Hodge rigid.

\begin{Thm}\label{hodgerigid}
Let $X$ be a $B$-stable Schubert variety in $G/P$. Assume that $\pi:Y\rightarrow X$ is a finite log-resolution. Then each homogeneous component of $c_{SM}(X^\circ)$ can be represented by an irreducible subvariety of $X$.
\end{Thm}

\begin{proof}
Let $c_{SM}(X^\circ)_k$ be the $k$-dimensional component of $c_{SM}(X^\circ)$. By the functoriality of CSM classes, it suffices to show that the pushforward of the class of any irreducible component of $\mathfrak{D}_k(\Lambda)$ is $0$, except for the unique component that is generically supported on $Y^\circ$.

Let $\Lambda$ be a general $(k+1)$-dimensional subspace of $\mathfrak{b}$. For each $B$-orbit $S$ in $Y$, set
$$D_k(\Lambda, S):=\{y\in S|\Lambda\cap \mathfrak{b}_y\neq 0\}.$$ 
Let $\mathfrak{I}$ be an irreducible component of $\mathfrak{D}_k(\Lambda)$ generically supported on a $B$-orbit $S\subset D$. By Theorem \ref{summary of Huh}, we have $\mathfrak{X}_S=\mathfrak{Y}_S$ and $\text{rank}(B)>\text{rank}(\text{Stab}_B(y))$ for $y\in S$. Notice that $\mathfrak{I}\subset \mathfrak{D}_k(\Lambda)\cap S=D_k(\Lambda,S)$. We claim that $\dim(\pi(D_k(\Lambda,S)))<\dim(D_k(\Lambda,S))$, and hence $\pi_*(\left[\mathfrak{I}\right])=0$.

    It suffices to show that for a general point $x\in \pi(D_k(\Lambda,S))$, we have $\dim(\pi^{-1}(x)\cap D_k(\Lambda,S))\geq 1.$ Let $y\in D_k(\Lambda,S)$ be such that $\pi(y)=x$. Since $\operatorname{Stab}_B(y)\subset \operatorname{Stab}_B(x)$, we have $\mathfrak{b}_y\subset \mathfrak{b}_x$. In particular, $\Lambda\cap \mathfrak{b}_y\neq 0$ implies $\Lambda\cap \mathfrak{b}_x\neq 0$.

Set $d_x:=\dim(\mathfrak{b}_x)$ and $d_y:=\dim(\mathfrak{b}_y)$. Define
\begin{eqnarray}
    \eta:\pi^{-1}(x)\cap S&\rightarrow& G(d_y,\mathfrak{b}_x)\nonumber\\
    y'&\mapsto& \mathfrak{b}_{y'}\nonumber
\end{eqnarray}
Let $\Psi$ be the image of $\eta$. Note that $\Psi$ is the orbit of $\mathfrak{b}_y$ under the adjoint action of $\operatorname{Stab}_B(x)$. Thus $\Psi\cong \text{Stab}_B(x)/N$, where $N$ is the normalizer of $\mathfrak{b}_y$ in $\operatorname{Stab}_B(x)$:
$$N:=\{s\in\text{Stab}_B(x)|Ad_s(\mathfrak{b}_y)=\mathfrak{b}_y\}.$$
In particular, $\Psi$ is locally closed and $$\dim(\Psi)=\dim(\text{Stab}_B(x))-\dim(N)=d_x-\dim(N).$$ 
Since $\operatorname{Stab}_B(y)\subset N$, we obtain
$$\dim(\Psi)\leq d_x-d_y,$$
with equality holds if $\eta$ is injective.

If $\dim(N)\geq d_y+1$, then $N\cdot y\subset \pi^{-1}(x)\cap D_k(\Lambda,S)$ and $\dim(N\cdot y)\geq 1$. This proves the claim in this case.

Now assume that $\dim(\Psi)=d_x-d_y$. Let $I$ be the Zariski closure of the locus
$$I^\circ:=\{(\bar{\xi},\mathfrak{b}_{y'}|\bar{\xi}\in\mathbb{P}(\mathfrak{b}_{y'}),y'\in\pi^{-1}(x)\cap S\}\subset \mathbb{P}(\mathfrak{b}_x)\times G(d_y,\mathfrak{b}_x).$$
Let $\pi_1:I\rightarrow \mathbb{P}(\mathfrak{b}_x)$ and $\pi_2:I\rightarrow\overline{\Psi}$ be the two projections. The fibers of $\pi_2$ have dimension $d_y-1$, and hence 
$$\dim(I)=\dim(\Psi)+d_y-1=d_x-1=\dim\mathbb{P}(\mathfrak{b}_x).$$
First assume that the fibers of $\pi_1$ have dimension at least $1$. Let $\xi_0\in \Lambda\cap \mathfrak{b}_y$. Since $\eta$ is injective, the set $\eta^{-1}(\pi_2(\pi_1^{-1}(\bar{\xi_0}))\cap \Psi)$ is contained in $\pi^{-1}(x)\cap D_k(\Lambda,S)$ and has dimension at least $1$.

Now assume that a general fiber of $\pi_1$ has dimension $0$. Then $\dim(\pi_1(I))=\dim\mathbb{P}(\mathfrak{b}_x)$, and hence $\pi_1|_{I^\circ}$ is dominant. Every $B$-orbit in $G/P$ contains a unique point fixed by the maximal torus $T$, and it follows by conjugation that every point $x\in X$ is fixed by some maximal torus $T_x\subset B$. Hence $\mathfrak{b}_x$ contains the Lie algebra of $T_x$, and in particular contains a regular element of $\mathfrak{b}$. Since the set of regular elements is open and dense in $\mathfrak{b}$, for a general point $x\in \pi(S)$, the intersection $\Lambda\cap \mathfrak{b}_x$ contains a regular element $\xi$.

By the dominance of $\pi_1|_{I^\circ}$, there exists $y\in \pi^{-1}(x)\cap S$ such that $\xi\in \mathfrak{b}_y$. Hence $y\in \pi^{-1}(x)\cap D_k(\Lambda,S)$. Let $T_x$ be a maximal torus in $\operatorname{Stab}_B(x)$ that centralizes $\xi$ under the adjoint action. Then for every $t\in T_x$, we have $t\cdot y\in \pi^{-1}(x)\cap D_k(\Lambda,S)$. Moreover, $\dim(T_x)=\text{rank}(B)>\text{rank}(\text{Stab}_B(y))$. Therefore $T_x\not\subset \text{Stab}_B(y)$ and hence $\dim(T_x\cdot y)\geq 1$. We conclude that $\pi^{-1}(x)\cap D_k(\Lambda,S)\geq 1$. This completes the proof.
\end{proof}

\begin{Rem}
Theorem \ref{hodgerigid} generalizes Huh's result in \cite{Huh} in the following sense: we do not require that, for every $y\in Y$, the isotropy algebra $\mathfrak{b}_y$ contain a regular element, a condition that does not hold in general.
\end{Rem}

The problem now reduces to finding finite log-resolutions of Schubert varieties. For Bott-Samelson resolutions, conditions {\rm (2)} and {\rm (3)} in Definition \ref{finiteresolution} always hold. In particular, spherical Bott-Samelson resolutions are finite log-resolutions.

\begin{Cor}
Let $\bar{w}$ be a reduced word for $w\in W^P$. If the corresponding Bott-Samelson variety $BS_{\bar{w}}$ is spherical, then each homogeneous component of $c_{SM}(X_P(w)^\circ)$ is Hodge rigid.
\end{Cor}

\section{Cominuscule Schubert varieties}\label{Cominuscule}
In this section, we consider cominuscule Schubert varieties. A cominuscule homogeneous variety is one of the following:
\begin{itemize}
\item Grassmannians;
\item smooth quadric hypersurfaces;
\item Lagrangian Grassmannians;
\item spinor varieties;
\item the two exceptional cases: the Cayley plane and the Freudenthal variety.
\end{itemize}

\subsection{Grassmannians}

Consider the Grassmannian $G(k,n)$. Let $F_\bullet$ be a complete flag of subspaces in $V$, and let $\Sigma_a(F_\bullet)$ be the corresponding Schubert variety. Choose a basis $e_1,\dots,e_n$ of $V$ such that $F_d=\text{span}(e_1,...,e_d)$ for $1\leq d\leq n$. With respect to this basis, let $G=SL(n)\cong SL(V)$. Let $T$ and $B$ be the subgroups of diagonal matrices and upper triangular matrices, respectively. Let $P_k$ be the maximal parabolic subgroup consisting of block upper triangular matrices with blocks of size $k$ and $n-k$. Then $G(k,n)\cong G/P_k$ and $\Sigma_a(F_\bullet)\cong X_{P_k}(w)$, where $w_i=a_i$ for $1\leq i\leq k$. In particular, $\Sigma_a(F_\bullet)$ is naturally equipped with a left action of $B$.

Let $\Gamma:=\Sigma_{a^\alpha}(F_\bullet)$ be the Schubert variety in $F(1,2,\dots,k;n)$ defined with respect to the same flag $F_\bullet$, where $\alpha=(1,\dots,k)$. Then $\Gamma$ also carries a natural left action of $B$. Let $\pi:\Gamma\rightarrow \Sigma_a(F_\bullet)$ be the canonical projection. Clearly, $\pi$ is $B$-equivariant. We claim that $\pi$ is a finite log resolution.
\begin{Prop}
The morphism $\pi:\Gamma\rightarrow\Sigma_a(F_\bullet)$ is a finite log resolution.
\end{Prop}
\begin{proof}
    \underline{Smoothness:} Let $$\pi_{\leq t}:F(1,...,k;n)\rightarrow F(1,...,t;n)$$ be the projection onto the first $t$ coordinates, for $1\le t\le k$. Note that $\pi_{\le 1}(\Gamma)\cong \mathbb{P}(F_{a_1})$, and for $2\le t\le k$, the variety $\pi_{\le t}(\Gamma)$ is a Grassmannian bundle over $\pi_{\le t-1}(\Gamma)$ with fibers isomorphic to $G(1,a_t-(t-1))$. Therefore, $\Gamma$ is smooth.
    
    \underline{Isomorphism over the Schubert cell:} Let $\Lambda\in \Sigma_a(F_\bullet)$ be a general point. Then $\pi^{-1}(\Lambda)$ consists of the unique point determined by the relations $\Lambda_t=\Lambda\cap F_{a_t}$.

    \underline{Finiteness of $B$-orbits:} Since $\Gamma$ is a Schubert variety in $F(1,\dots,k;n)$, the finiteness of $B$-orbits follows from the Bruhat decomposition of $F(1,\dots,k;n)$.

    \underline{SNC divisor:} Let $D:=\pi^{-1}(\partial\Sigma_{a})$ and $D_t:=\pi_{\leq t}(D)$, $1\leq t\leq k$. Note that $\pi_{\le t}(\Gamma)$ is a Schubert variety in $F(1,\dots,t;n)$ and that $D_t$ coincides with $\partial \pi_{\le t}(\Gamma)$. We use induction on $t$. When $t=1$, we have $D_1\cong \mathbb{P}^{a_1-2}$, which is smooth and irreducible. For $2\le t\le k$,
    $$D_t=\pi_{\leq t}(\pi^{-1}_{\leq t-1}(D_{t-1}))\cup\tilde{D_t}, $$
    where $\tilde{D_t}$ is a projective bundle over $\pi_{\leq t-1}(\Gamma)$ modeled on $\mathbb{P}^{a_t-t-1}$. It follows that $D_t$ is an SNC divisor in $\pi_{\leq t}(\Gamma)$ provided that $D_{t-1}$ is a SNC divisor in $\pi_{\leq t-1}(\Gamma)$. By induction, we conclude that $D$ is a SNC divisor in $\Gamma$.
\end{proof}

\subsection{Smooth quadric hypersurfaces}Let $Q$ be a smooth quadric hypersurface in $\mathbb{P}(V)$. Let $F_1\subset...\subset F_{\left[n/2\right]}$ be an isotropic flag. Then the Schubert varieties of $Q$ are $\Sigma_{a;}\cong\mathbb{P}(F_a)$, $1\leq a\leq \left[\frac{n}{2}\right]$, and $\Sigma_{;b}\cong\mathbb{P}(F_b^\perp)\cap Q$, $0\leq b\leq \left[\frac{n}{2}\right]-1$. Since there is only one Schubert class in each degree, except in the middle degree when $n$ is even, we may construct irreducible representatives of the CSM classes directly. We first recall the following lemma on the flexibility of Schubert classes.

\begin{Lem}\cite[Corollary 3.4]{Flexibility}\label{flex}
    Let $\sigma$ be a Schubert class on a smooth quadric hypersurface. Assume $\sigma$ is not the fundamental class, the class of a point, or the class of a maximal isotropic subspace when $n$ is even. Then $m\sigma$ can be represented by an irreducible subvariety for every $m\geq 1$.
\end{Lem}

\begin{Prop}
 Let $\Sigma$ be a Schubert variety in $Q$. For each $0\le k\le \dim(\Sigma)$, there exists an irreducible subvariety of $\Sigma$ representing $c_{SM}(\Sigma^\circ)_k$.
\end{Prop}
\begin{proof}
    For $k\neq 0,\frac{n}{2},n-2$, the statement follows from Lemma \ref{flex}.

    For $k=0$, note that $\chi(\Sigma^\circ)=1$, and hence $c_{SM}(\Sigma^\circ)_0=\sigma_{1;}$, which can be represented by the Schubert variety $\Sigma_{1;}$.

    For $k=n-2$ and $\Sigma=Q$, the statement is trivial, since the coefficient of the top-dimensional component in $c_{SM}(\Sigma^\circ)$ is always $1$.

    Now assume that $n$ is even and $k=\frac{n}{2}$. Let $\gamma$ be an involution that interchanges the two irreducible components of the space of maximal isotropic subspaces. Then $\gamma_*$ interchanges the classes $\sigma_{n/2;}$ and $\sigma_{;n/2-1}$, while fixing all other Schubert classes. By the functoriality of CSM classes, the coefficients of $\sigma_{n/2;}$ and $\sigma_{;n/2-1}$ in $c_{SM}(\Sigma^\circ)$ must therefore be equal. Hence it suffices to construct an irreducible subvariety representing the class $m(\sigma_{n/2;}+\sigma_{;n/2-1})$ for any $m\ge 1$. For example, this class can be represented by the intersection of $Q$ with a general hypersurface of degree $m$ contained in $\mathbb{P}(F_{n/2-2}^\perp)$.
    
\end{proof}

\subsection{Lagrangian Grassmannians}

Consider the Lagrangian Grassmannian $SG(k,2k)$ with $n=2k$. Let $F_1\subset...\subset F_{\left[n/2\right]}$ be an isotropic flag in $V$, and let $\Sigma_{a;b}(F_\bullet)$ be the corresponding Schubert variety. Choose a basis $e_1,...,e_n$ of $V$ such that $F_i=\text{span}(e_1,...,e_i)$ for $1\leq i\leq k$ and $F_j^\perp=\text{span}(e_1,...,e_{2k-j})$ for $0\leq j\leq k-1$. With respect to this basis, let $G=Sp(n)\cong Sp(V)$. Let $T$ and $B$ be the subgroups of diagonal matrices and upper triangular matrices in $G$, respectively.

Note that for a Schubert index $(a;b)$, the sequence $b$ is uniquely determined by $a$ together with the condition $a_i\neq b_j$ for all $i,j$. Let $\Gamma$ be the Schubert variety $\Sigma_{a^\alpha;b^\beta}(F_\bullet)$ in $SF(1,2,\dots,s,k;n)$, where $\alpha=(1,\dots,s)$ and $\beta=(s+1,\dots,s+1)$. Then $\Gamma$ is naturally equipped with a left action of $B$, and the canonical projection $\pi:\Gamma\rightarrow \Sigma_{a;b}(F_\bullet)$ is $B$-equivariant. We claim that $\pi$ is a finite log resolution.
\begin{Prop}\label{LG}
The morphism $\pi:\Gamma\rightarrow\Sigma_{a;b}(F_\bullet)$ is a finite log resolution.
\end{Prop}
\begin{proof}
    \underline{Smoothness:} If $s=|a|=0$, then $\Sigma_{a;b}(F_\bullet)=SG(k,2k)$ is already smooth. Now assume that $s\ge 1$. Let $\pi_{\leq t}:SF(1,...,s,k;n)\rightarrow F(1,...,t;n)$ be the projection onto the first $t$ coordinates, $1\le t\le s$. Note that $\pi_{\le 1}(\Gamma)\cong \mathbb{P}(F_{a_1})$, and for $2\le t\le s$, the variety $\pi_{\le t}(\Gamma)$ is a Grassmannian bundle over $\pi_{\le t-1}(\Gamma)$ with fibers isomorphic to $G(1,a_t-(t-1))$. Moreover, $\Gamma$ is a projective bundle over $\pi_{\le s}(\Gamma)$ with fibers isomorphic to $SG(k-s,n-2s)$. Therefore, $\Gamma$ is smooth.
    
    \underline{Isomorphism over the Schubert cell:}
Let $\Lambda\in \Sigma_{a;b}(F_\bullet)$ be a general point. Then $\pi^{-1}(\Lambda)$ consists of the unique point determined by the relations $\Lambda_t=\Lambda\cap F_{a_t}$, $1\leq t\leq s$.

    \underline{Finiteness of $B$-orbits:} Since $\Gamma$ is a Schubert variety in $SF(1,\dots,s,k;n)$, the finiteness of $B$-orbits follows from the Bruhat decomposition of $SF(1,\dots,s,k;n)$.

    \underline{SNC divisor:} Let $D:=\pi^{-1}(\partial\Sigma_{a;b})$ and $D_t:=\pi_{\leq t}(D)$, $1\leq t\leq s$. Note that $\pi_{\le t}(\Gamma)$ is a Schubert variety in $F(1,\dots,t;n)$ and that $D_t$ coincides with $\partial \pi_{\le t}(\Gamma)$. We use induction on $t$. When $t=1$, we have $D_1\cong\mathbb{P}^{a_1-2}$ is smooth and irreducible. For $2\leq t\leq s$,
    $$D_{t}=\pi_{\leq t}(\pi^{-1}_{\leq t-1}(D_{t-1}))\cup\tilde{D_t}, $$
    where $\widetilde{D}_t$ is a projective bundle over $\pi_{\le t-1}(\Gamma)$ modeled on $\mathbb{P}^{a_t-t-1}$. Therefore, $D_t$ is an SNC divisor in $\pi_{\le t}(\Gamma)$ if $D_{t-1}$ is an SNC divisor in $\pi_{\le t-1}(\Gamma)$. By induction, $D_s$ is an SNC divisor in $\pi_{\le s}(\Gamma)$. Also note that
    $$D=\pi_{\leq s}^{-1}(D_s),$$
    and therefore $D$ is an SNC divisor in $\Gamma$.
\end{proof}

\subsection{Spinor varieties}

The same construction for Lagrangian Grassmannians also works for Spinor varieties. Let $\Sigma_{a;b}(F_\bullet)$ be a Schubert variety in $OG(k,2k)$. Note that $b$ is uniquely determined by $a$ together with the condition $a_i\neq b_j$ for all $i,j$. Let $\Gamma$ be the Schubert variety $\Sigma_{a^\alpha;b^\beta}(F_\bullet)$ in $OF(1,2,\dots,s,k;n)$, where $\alpha=(1,\dots,s)$ and $\beta=(s+1,\dots,s+1)$.
\begin{Prop}\label{spinor}
The canonical projection $\pi:\Gamma\rightarrow\Sigma_{a;b}(F_\bullet)$ is a finite log resolution.
\end{Prop}
\begin{proof}
The proof is the same as that of Proposition \ref{LG}.
\end{proof}

\subsection{Cayley plane}

Consider the Cayley plane $\mathbb{OP}^2=E_6/P_6$. The numbering of the simple roots and their incidence relations are illustrated in the Dynkin diagrams (see Figure \ref{Dynkin1}).

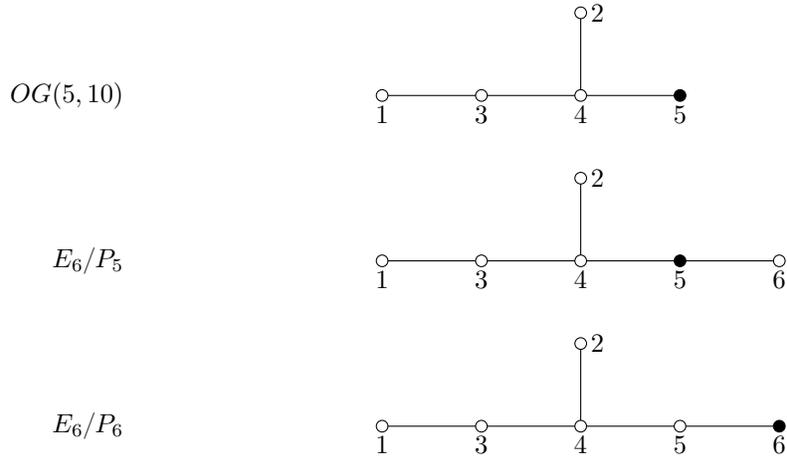
\begin{figure}[ht]
\centering
\caption{Marked Dynkin diagrams}
\begin{tikzpicture}[scale=1.1, every node/.style={font=\small}]\label{Dynkin1}

\node[left] at (0,4) {$OG(5,10)$};
\node[left] at (0,2) {$E_6/P_5$};
\node[left] at (0,0) {$E_6/P_6$};

\begin{scope}[shift={(3,4)}]
  \draw (0,0)--(1.2,0)--(2.4,0)--(3.6,0);
  \draw (2.4,0)--(2.4,1.0);

  \filldraw[fill=white] (0,0) circle (2pt);
  \filldraw[fill=white] (1.2,0) circle (2pt);
  \filldraw[fill=white] (2.4,0) circle (2pt);
  \filldraw[fill=black] (3.6,0) circle (2pt);
  \filldraw[fill=white] (2.4,1.0) circle (2pt);

  \node[below] at (0,0) {$1$};
  \node[below] at (1.2,0) {$3$};
  \node[below] at (2.4,0) {$4$};
  \node[below] at (3.6,0) {$5$};
  \node[right] at (2.4,1.0) {$2$};
\end{scope}

\begin{scope}[shift={(3,2)}]
  \draw (0,0)--(1.2,0)--(2.4,0)--(3.6,0)--(4.8,0);
  \draw (2.4,0)--(2.4,1.0);

  \filldraw[fill=white] (0,0) circle (2pt);
  \filldraw[fill=white] (1.2,0) circle (2pt);
  \filldraw[fill=white] (2.4,0) circle (2pt);
  \filldraw[fill=black] (3.6,0) circle (2pt);
  \filldraw[fill=white] (4.8,0) circle (2pt);
  \filldraw[fill=white] (2.4,1.0) circle (2pt);

  \node[below] at (0,0) {$1$};
  \node[below] at (1.2,0) {$3$};
  \node[below] at (2.4,0) {$4$};
  \node[below] at (3.6,0) {$5$};
  \node[below] at (4.8,0) {$6$};
  \node[right] at (2.4,1.0) {$2$};
\end{scope}

\begin{scope}[shift={(3,0)}]
  \draw (0,0)--(1.2,0)--(2.4,0)--(3.6,0)--(4.8,0);
  \draw (2.4,0)--(2.4,1.0);

  \filldraw[fill=white] (0,0) circle (2pt);
  \filldraw[fill=white] (1.2,0) circle (2pt);
  \filldraw[fill=white] (2.4,0) circle (2pt);
  \filldraw[fill=white] (3.6,0) circle (2pt);
  \filldraw[fill=black] (4.8,0) circle (2pt);
  \filldraw[fill=white] (2.4,1.0) circle (2pt);

  \node[below] at (0,0) {$1$};
  \node[below] at (1.2,0) {$3$};
  \node[below] at (2.4,0) {$4$};
  \node[below] at (3.6,0) {$5$};
  \node[below] at (4.8,0) {$6$};
  \node[right] at (2.4,1.0) {$2$};
\end{scope}

\end{tikzpicture}
\end{figure}

The Schubert varieties are indexed by the minimal representatives $w\in W^{P_6}=W/W_{P_6}$. In particular, there are $27$ Schubert varieties in $E_6/P_6$ (see Table \ref{table1}). In the table, each row represents a Schubert variety in $E_6/P_6$. The first column gives a reduced word for $w$ as a product of simple reflections. For example, $(56)$ stands for $w=s_5s_6$. The second column is the dimension of the Schubert variety $X_{P_6}(w)$, and the third column is its degree.
\begin{table}[ht]
\centering
\caption{Schubert varieties in the Cayley plane.}\label{table1}
\begin{minipage}{0.46\textwidth}
\centering
\begin{tabular}{|c|c|c|}

\hline
$w$ & $\dim$ & deg   \\
\hline
e  & 0         & 1             \\
6  & 1         & 1           \\
56  & 2        & 1            \\
456  & 3       & 1            \\
2456  & 4      & 1             \\
3456  & 4      & 1            \\
13456  & 5    & 1   \\
23456  & 5    & 2             \\
123456  & 6   &  3   \\
423456  & 6   & 2    \\
1423456  & 7   & 5   \\
5423456  & 7   & 2      \\
31423456  &  8 &  5    \\
15423456 & 8  & 7   \\
\hline
\end{tabular}
\end{minipage}
\hfill
\begin{minipage}{0.46\textwidth}
\centering
\begin{tabular}{|c|c|c|}
\hline
$w$ & dim & deg   \\
\hline
65423456   & 8       & 2             \\
315423456  & 9      & 12   \\
165423456  & 9      & 9    \\
4315423456 & 10     & 12    \\
3165423456 & 10     & 21    \\
24315423456 & 11    & 12       \\
43165423456 & 11    & 33   \\
243165423456 & 12   & 45    \\
543165423456 & 12   & 33     \\
2543165423456 & 13 &  78   \\
42543165423456 & 14 & 78     \\
342543165423456 & 15 & 78    \\
1342543165423456 & 16 & 78    \\
\hline
\end{tabular}
\end{minipage}
\end{table}

A standard way to study the geometry of Schubert varieties in the Cayley plane is via the Tits correspondence. Consider the diagram \[
\begin{tikzcd}
& E_6/(P_5\cap P_6) \arrow[dl, "\tau"'] \arrow[dr, "\eta"] & \\
E_6/P_5 & & E_6/P_6.
\end{tikzcd}
\]
Given a subset $U\subset E_6/P_5$, define its Tits transform by $\mathcal{T}(U):=\eta(\tau^{-1}(U))$. In particular, the Tits transform of a point in $E_6/P_5$ is a line in $E_6/P_6$. Let $x$ be the unique $B$-fixed point in $E_6/P_6$. Then the lines in $E_6/P_6$ passing through $x$ are parametrized by the spinor variety $OG(5,10)\subset E_6/P_5$. Therefore, $\mathcal{T}(OG(5,10))=X_{P_6}(24315423456)$ is isomorphic to the cone over $OG(5,10)$, and every subvariety of $X_{P_6}(24315423456)$ arises in this way. We prove that the finite log-resolution constructed for Schubert varieties in $OG(5,10)$ can be lifted to a finite log-resolution of Schubert varieties in $E_6/P_6$.

\begin{Prop}\label{Cayley}
Let $X_{P_6}(w)$ be a Schubert variety in $E_6/P_6$ with $w\leq (24315423456)$. Then there exists a finite log-resolution $\pi:Y\rightarrow X_{P_6}(w)$. In particular, $Y$ can be chosen to be a Schubert variety in $E_6/P$ for some parabolic subgroup $P\subset P_6$, and $\pi$ is induced by the natural projection $E_6/P\rightarrow E_6/P_6$.
\end{Prop}
\begin{proof}
    If $w=0$, then the statement is trivial. Assume $w\neq 0$. Then $X_{P_6}(w)=\mathcal{T}(\Sigma)$ for some Schubert variety $\Sigma\cong \Sigma_{a;b}\subset OG(5,10)\hookrightarrow E_6/P_5$. By Proposition \ref{spinor}, there exists a finite log-resolution $\pi':\Gamma\rightarrow \Sigma_{a;b}$ where $\Gamma$ is a Schubert variety in $OF(1,\dots,s,5;n)$ for some $0\leq s\leq 5$. Under the embedding $OF(1,...,s,5;n)\cong D_5/P_I\hookrightarrow E_6/P_I$, $\Gamma$ can be identified with a Schubert variety in $E_6/P_I$, and $\pi'$ is $B_{E_6}$-equivariant and induced by the projection $E_6/P_I\rightarrow E_6/P_5$.

Set $Y:=\Gamma\times_{\Sigma}\tau^{-1}(\Sigma)$. Since $\Gamma$ is a smooth Schubert variety and $\tau$ is a smooth homogeneous fibration, $Y$ is a smooth Schubert variety in $E_6/P_{I\cup{5,6}}$. In particular, $Y$ has finitely many $B_{E_6}$-orbits by the Bruhat decomposition. We have the following commutative diagram:

\[
\begin{tikzcd}[column sep=large, row sep=large]
Y:=\Gamma\times_\Sigma \tau^{-1}(\Sigma) \arrow[r, "\pi_2"] \arrow[d, "\pi_1"'] & \tau^{-1}(\Sigma) \arrow[r, "\eta"] \arrow[d, "\tau"] & \mathcal{T}(\Sigma)=X_{P_6}(w) \\
\Gamma \arrow[r, "\pi'"] & \Sigma &
\end{tikzcd}
\]
Set $\pi:=\eta\circ\pi_2$. Clearly, $\pi$ is $B_{E_6}$-equivariant. Since $\tau^{-1}(\Sigma)\to\mathcal{T}(\Sigma)$ is an isomorphism over $\mathcal{T}(\Sigma)^\circ$ and $\pi':\Gamma\to \Sigma$ is an isomorphism over $\Sigma^\circ$, it follows that $\pi:Y\rightarrow\mathcal{T}(\Sigma)$ is an isomorphism over $\mathcal{T}(\Sigma)^\circ$. Therefore $\pi$ is birational.

Let $D':=\pi'^{-1}(\partial \Sigma)$. Since $\pi_1$ is smooth and $D'$ is an SNC divisor on $\Gamma$, the pullback $D:=\pi_1^{-1}(D')$ is an SNC divisor on $Y$. Note that $D=\pi^{-1}(\partial\mathcal{T}(\Sigma))$ and hence $\pi:Y\rightarrow \mathcal{T}(\Sigma)=X_{P_6}(w)$ is a finite log-resolution.
\end{proof}

\subsection{Freudenthal variety}

Consider the Freudenthal variety $E_7/P_7$. There are $56$ Schubert varieties, indexed by the minimal representatives $w\in W^{P_7}$ (see Table \ref{table2}). In the table, each row represents a Schubert variety in $E_7/P_7$. The first column gives a reduced word for $w$ as a product of simple reflections. For example, $(67)$ stands for $w=s_6s_7$. The second column is the dimension of the Schubert variety $X_{P_7}(w)$, and the third column is its degree. One Schubert variety of particular interest is $X_{P_7}(243154234567)$, which is isomorphic to the cone over the Cayley plane.

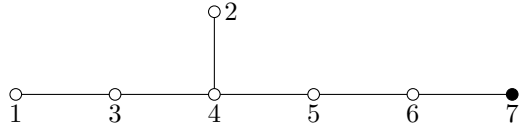
\begin{figure}[ht]
\centering
\caption{Marked Dynkin diagram for $E_7/P_7$}
\begin{tikzpicture}[scale=1.1, every node/.style={font=\small}]

\begin{scope}[shift={(3,0)}]
  \draw (0,0)--(1.2,0)--(2.4,0)--(3.6,0)--(4.8,0)--(6.0,0);
  \draw (2.4,0)--(2.4,1.0);

  \filldraw[fill=white] (0,0) circle (2pt);
  \filldraw[fill=white] (1.2,0) circle (2pt);
  \filldraw[fill=white] (2.4,0) circle (2pt);
  \filldraw[fill=white] (3.6,0) circle (2pt);
  \filldraw[fill=white] (4.8,0) circle (2pt);
  \filldraw[fill=black] (6.0,0) circle (2pt);
  \filldraw[fill=white] (2.4,1.0) circle (2pt);

  \node[below] at (0,0) {$1$};
  \node[below] at (1.2,0) {$3$};
  \node[below] at (2.4,0) {$4$};
  \node[below] at (3.6,0) {$5$};
  \node[below] at (4.8,0) {$6$};
  \node[below] at (6.0,0) {$7$};
  \node[right] at (2.4,1.0) {$2$};
\end{scope}

\end{tikzpicture}
\end{figure}

\begin{table}[ht]
\centering
\caption{Schubert varieties in the Freudenthal variety.}\label{table2}

\begin{minipage}[t]{0.42\textwidth}
\centering
\begin{tabular}{|c|c|c|}
\hline
$w$ & dim & deg   \\
\hline
 e            & 0   &    1      \\
 7            & 1   &    1      \\
 67           & 2   &     1     \\
 567          & 3   &      1    \\
 4567         & 4  &       1   \\
 24567        & 5   &    1      \\
 34567        & 5  &     1     \\
 134567       & 6   &  1\\
 234567       & 6   &   2       \\
 1234567      & 7  & 3 \\
 4234567      & 7  & 2 \\
 14234567     & 8 & 5\\
 54234567     & 8  & 2 \\
 314234567    & 9   &5 \\
 154234567    & 9 & 7 \\
 654234567    & 9  &  2     \\
 3154234567   & 10 & 12\\
 1654234567   & 10 & 9 \\
 7654324567   & 10   &  2        \\
 43154234567  & 11  & 12 \\
 31654234567  & 11 & 21 \\
 17654234567  & 11  & 11 \\
 243154234567 & 12   & 12 \\
\hline
\end{tabular}
\end{minipage}
\hfill
\begin{minipage}[t]{0.53\textwidth}
\centering
\begin{tabular}{|c|c|c|}
\hline
 $w$ & dim & deg \\
\hline
 431654234567     & 12 & 33 \\
 317765234567     & 12   & 32 \\
 2431654234567    & 13  & 45 \\
 5431654234567    & 13   & 33 \\
 4317654234567    & 13  & 65 \\
 25431654234567   & 14  & 78 \\
 24317654234567   & 14   & 110 \\
 54317654234567   & 14  & 98 \\
 425431654234567  & 15   & 78 \\
 254317654234567  & 15 & 286 \\
 654317654234567  & 15    & 98 \\
 3425431654234567 & 16   & 78 \\
 4254317654234567 & 16  & 364 \\
 2654317654234567 & 16   & 384 \\
 13425431654234567 & 17   & 78 \\
 34254317654234567 & 17 & 442 \\
 42654317654234567 & 17 & 748 \\
 134254317654234567 & 18 & 520 \\
 432654317654234567 & 18 & 1190 \\
 542654317654234567 & 18   & 748 \\
 1342654317654234567 & 19 & 1710 \\
 3542654317654234567 & 19 & 1938 \\
\hline
\end{tabular}
\end{minipage}

\vspace{1em}

\begin{tabular}{|c|c|c|}
\hline
 $w$ & dim & deg \\
\hline
 13542654317654234567   & 20  & 3648 \\
 43542654317654234567   & 20  & 1938 \\
 143542654137654234567  & 21 & 5586 \\
 243542654317654234567  & 21   & 1938 \\
 1243542654317654234567 & 22  & 5586 \\
 3143542654317654234567 & 22   & 7524 \\
 23143542654317654234567 & 23 & 13110 \\
 423143542654317654234567 & 24 & 13110 \\
 5423143542654317654234567 & 25 & 13110 \\
 65423143542654317654234567 & 26 & 13110 \\
765423143542654317654234567 & 27 & 13110 \\
\hline
\end{tabular}

\end{table}

\begin{Prop}
Let $X_{P_7}(w)$ be a Schubert variety in $E_7/P_7$ with $w\leq (243154234567)$. Then there exists a finite log resolution $\pi:Y\rightarrow X_{P_7}(w)$. In particular, $Y$ can be chosen to be a Schubert variety in $E_7/P$ for some parabolic $P\subset P_7$, and $\pi$ is induced by the natural projection $E_7/P\rightarrow E_7/P_7$.
\end{Prop}
\begin{proof}
Consider the Tits correspondence 
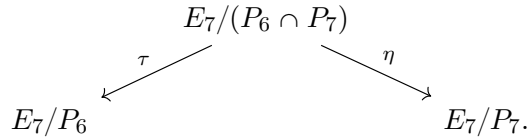
\[
\begin{tikzcd}
& E_7/(P_6\cap P_7) \arrow[dl, "\tau"'] \arrow[dr, "\eta"] & \\
E_7/P_6 & & E_7/P_7.
\end{tikzcd}
\]
For every $w\leq (243154234567)$, we have $X_{P_7}(w)=\mathcal{T}(X):=\eta(\tau^{-1}(X))$ for some Schubert variety $X$ in $E_6/P_6\subset E_7/P_6$. This correspondence can be obtained by comparing the dimensions and degrees in Table \ref{table1} and Table \ref{table2}. By Proposition \ref{Cayley}, there exists a finite log-resolution $Y'\rightarrow X$. The conclusion then follows by the same argument as in the proof of Proposition \ref{Cayley}.
\end{proof}

\section{Type A flag varieties}\label{typeA}
In this section, we consider Schubert varieties and Schubert cells in $F(k_1,\dots,k_h;n)$, which is isomorphic to $G/P$ with $G=SL(V)$. 

Fix an ordered basis $e_1,\dots,e_n$ of $V$. Let $F_i$ be the vector subspace spanned by $e_1,\dots,e_i$, for $1\le i\le n$. Let $B\subset SL(n;\mathbb{C})$ be the Borel subgroup consisting of upper triangular matrices, and let $T\subset B$ be the maximal torus consisting of diagonal matrices.

\subsection{Resolution of singularities}

Let $\Sigma_{a^\alpha}(F_\bullet)$ be the Schubert variety in $F(k_1,\dots,k_h;n)$. The associated rank matrix $M$ is the $h\times k_h$ matrix with entries $M_{ti}=\mu_{i,t}$, $1\leq i\leq k_h$, $1\leq t\leq h$. Consider the following diagram
\begin{center}
\begin{tabular}{ c c c c c c c }
 $W^{1,1}$&  & ... & &$W^{1,k_h-1}$& &$\Lambda_1$   \\ 
         $\vdots$ & & $\vdots$ &  &$\vdots$ &  &$\vdots$\\
                 $W^{h,1}$ & & ... &  & $W^{h,k_h-1}$&    &$\Lambda_h$\\
                 \hline
 $F_{a_1}$ & $\subset$ & ...&$\subset$&$F_{a_{k_h}-1}$ &$\subset$ & $F_{a_{k_h}}$   
\end{tabular}
\end{center}
where $W^{t,i}$ are subspaces of dimension $\mu_{i,t}$, for $1\le t\le h$ and $1\le i\le k_h-1$, and $\Lambda_t$ are subspaces of dimension $k_t$, for $1\le t\le h$, such that each row and each column form a partial flag of subspaces. For each row, if $W^{t,i}=W^{t,i'}$ for some $i<i'$ (equivalently, if $\mu_{i,t}=\mu_{i',t}$), we delete the redundant entries $W^{t,i'}$ and keep only the leftmost one $W^{t,i}$. There are exactly $k_t$ non-repeated entries in the $t$-th row. Let $\Gamma$ be the collection of the resulting diagrams. Then $\Gamma$ can be viewed as a subset of a product of Grassmannians, and it maps onto $\Sigma_{a^\alpha}(F_\bullet)$ under the projection sending each diagram to $(\Lambda_1,\dots,\Lambda_h)$. Note that $\Gamma$ admits a left action of $B$ on each component, and the projection is $B$-equivariant.

\begin{Thm}\label{resolutionA}
    The natural projection $\pi:\Gamma\rightarrow\Sigma_{a^\alpha}(F_\bullet)$ is a resolution of singularities.
\end{Thm}
\begin{proof}
    We show that $\Gamma$ can be realized as an iterated tower of Grassmannians. We use induction, starting from $W^{h,1}$ and proceeding from left to right, then from bottom to top. More precisely, the order is given by $(t,i)>(t',i')$ if $t<t'$ or if $t=t'$ and $i>i'$.

    When $(t,i)=(h,1)$, the locus of $W^{h,1}$ is parametrized by the Grassmannian $G(\mu_{1,h},F_{a_1})$, which is smooth and irreducible.

    Now let $(t,i)>(h,1)$, and let $\Gamma^{<(t,i)}$ be the image of $\Gamma$ under the projection forgetting all entries $W^{t',i'}$ with $(t',i')\geq (t,i)$. Inductively, assume that $\Gamma^{<(t,i)}$ is an iterated tower of Grassmannians. If $\mu_{i,t}>1$, let $i_0$ be the minimal integer such that $\mu_{i_0,t}=\mu_{i,t}-1$. Otherwise, set $i_0=0$ and $W^{t,0}=0$. If $t=h$, set $W^{h+1,i}=F_{a_i}$. Then, over every point of $\Gamma^{<(t,i)}$, the subspace $W^{t,i}$ can be any subspace of dimension $\mu_{i,t}$ satisfying
    $$W^{t,i_0}\subset W^{t,i}\subset W^{t+1,i}.$$
    Therefore, we obtain a $G(\mu_{i,t}-\mu_{i_0,t},\mu_{i,t+1}-\mu_{i_0,t})$-bundle over $\Gamma^{<(t,i)}$. We conclude by induction that $\Gamma$ is an iterated tower of Grassmannians. In particular, $\Gamma$ is smooth and irreducible.

    Note that $\pi$ is an isomorphism over the Schubert cell $\Sigma_{a^\alpha}^\circ(F_\bullet)$. Indeed, for every $(\Lambda_1,\dots,\Lambda_h)\in \Sigma_{a^\alpha}^\circ(F_\bullet)$, the fiber $\pi^{-1}(\Lambda_1,\dots,\Lambda_h)$ is uniquely determined by the relations $W^{t,i}=\Lambda_t\cap F_{a_i}$. Therefore, $\pi:\Gamma\to\Sigma_{a^\alpha}(F_\bullet)$ is a resolution of singularities.
\end{proof}

\begin{Thm}\label{4.4}
Assume that $h=2$ and $k_1=1$. Then the resolution $\pi:\Gamma\rightarrow\Sigma_{a^\alpha}(F_\bullet)$ is a finite log resolution.
\end{Thm}
\begin{proof}
    By Theorem \ref{resolutionA}, it suffices to show that $\Gamma$ has finitely many $B$-orbits and that $\Gamma\backslash\pi^{-1}(\Sigma_{a^\alpha}^\circ)$ is an SNC divisor.

    Note that there is a natural projection $\pi':\Gamma\rightarrow F(1,2,...,k_2;n)$ to the bottom row, whose image is the Schubert variety $\Sigma_{a^\beta}(F_\bullet)$, where $\beta=(1,2,\dots,k_2)$. We claim that every $B$-orbit $C$ (that is, Schubert cell) in $\Sigma_{a^\beta}(F_\bullet)$ lifts to finitely many $B$-orbits in $\Gamma$. 
    
    Let $a^1=a_{i_0}$ be the unique sub-index such that $\alpha_{i_0}=1$. Let $x\in C$ be the unique $T$-fixed point in $C$, whose entries are subspaces spanned by subsets of the basis vectors ${e_1,\dots,e_n}$ of $V$. Let $x_{i_0}$ be the $i_0$-th entry of $x$. Then the natural action of $T$ on $G(1,x_{i_0})$ has $2^{i_0}-1$ orbits. Since the stabilizer of $x$ contains $T$, there are only finitely many $B$-orbits in $\Gamma$ lying over $C$. By the Bruhat decomposition, there are only finitely many Schubert cells in $\Sigma_{a^\beta}(F_\bullet)$, and therefore $\Gamma$ has only finitely many $B$-orbits. 

    Let $D_1$ be the union of Schubert divisors of $\Sigma_{a^\beta}(F_\bullet)$. It is known that $D_1$ is an SNC divisor. Notice that $\Gamma\backslash\pi^{-1}(\Sigma_{a^\alpha}^\circ)=\pi'^{-1}(D_1)\cup D_2$ where $D_2$ is a $\mathbb{P}^1$-bundle over $\Sigma_{a^\beta}(F_\bullet)$. Therefore $\pi'^{-1}(D_1)\cup D_2$ is an SNC divisor in $\Gamma$.   
\end{proof}
\begin{Rem}
    In fact, the resolution constructed in Theorem \ref{resolutionA} is isomorphic to the classical Bott--Samelson resolution, for which the SNC divisor property is standard.
\end{Rem}
\begin{Cor}
    The CSM classes of Schubert cells in $F(1,k_2;n)$ are Hodge rigid.
\end{Cor}
\begin{proof}
    This follows from Theorem \ref{hodgerigid} and Theorem \ref{4.4}.
\end{proof}

\section{Symplectic Grassmannians}\label{SG}
In this section, we consider Schubert varieties in symplectic Grassmannians. In particular, we construct resolutions of singularities of these Schubert varieties and prove that, under certain conditions, these resolutions are finite log resolutions. As an application, we also describe the singular locus of Schubert varieties in symplectic Grassmannians.

\subsection{Resolution of singularities}

We begin with an example.
\begin{Ex}
    Let $\Sigma_{;2}$ be the Schubert variety in $SG(1,n)$, $n\geq 6$, defined by the sequence $F_2^\perp$. Consider the following locus in $SF(1,3;n)$:
    $$\Gamma:=\{(\Lambda_1,\Lambda_2)\in SF(1,3;n)|\Lambda_1\subset \Lambda_2, F_2\subset\Lambda_2\subset F_2^\perp\}.$$
    We claim that the forgetful map $\pi:\Gamma\rightarrow \Sigma_{;2}$, $(\Lambda_1,\Lambda_2)\mapsto \Lambda_1$ is a resolution of singularities. Indeed, $\Gamma$ maps onto $\Sigma_{;2}$, and for a general point $\Lambda_1\in \Sigma_{;2}$, the fiber of $\pi$ is uniquely determined by taking $\Lambda_2$ to be the span of $F_2$ and $\Lambda_1$. Let $\pi_2:SF(1,3;n)\rightarrow SG(3,n)$ be the second projection. Then $\pi_2(\Gamma)$ is isomorphic to $SG(1,n-2)$, and every fiber of $\pi_2$ is isomorphic to $\mathbb{P}^2$. Hence $\Gamma$ is smooth. By Zariski's Main Theorem, $\pi$ is a resolution of singularities. It also follows that the singular locus of $\Sigma_{;2}$ is $F_2$.
\end{Ex}

\begin{Def}\label{resolution}
    Let $\Sigma_{a;b}(F_\bullet)$ be the Schubert variety in $SG(k,n)$ defined by an isotropic flag $F_\bullet$. For each $1\leq j\leq k-s$, set $$x_j:=\#\{i|a_i\leq b_j\}.$$ 
    For $1\leq \tau\leq k$, set $z_\tau:=\#\{j|b_j<a_\tau\}$ if $\tau\leq s$, and set $z_\tau:=k-\tau+1$ if $\tau>s$. Define
\begin{eqnarray}
    \Gamma_{a;b}(F_\bullet):=\{(\Lambda_1,...,\Lambda_k,W^{1,1},...,W^{1,z_1},W^{2,1},...,W^{k,z_k})&|& F_{z_\tau}\subset W^{\tau,z_\tau}\subset F_{a_\tau}\text{ or }F_{b_{k-\tau+1}}^\perp,\nonumber\\
    & &W^{\tau-1,j}\subset W^{\tau,j}\subset W^{\tau,j+1}\text{ for }1\leq j\leq z_\tau-1,\nonumber\\
    & &\Lambda_{\tau-1}\subset\Lambda_{\tau}\subset W^{\tau,z_\tau}\}\nonumber
\end{eqnarray}
\end{Def}
\begin{Rem}\label{tableS}
    An element in $\Gamma_{a;b}(F_\bullet)$ can be arranged into the following table
\begin{center}
\begin{tabular}{ c c c c c c c c c}
 &  & $F_{b_1}$ & &...& &$F_{b_{k-s}}$ & &  \\ 
  & & $\cap$ & & & &$\cap$ & &  \\  
 $\Lambda_1$ & $\subset$ & $W^{1,1}$ & $\subset$ & ...&$\subset$&$W^{1,k-s}$ &$\subset$  &$F_{a_1}$   \\
     $\cap$ & & $\cap$ & & & &$\cap$ & &$\cap$\\
         $\vdots$ & & $\vdots$ & & & &$\vdots$ &  &$\vdots$\\
                 $\cap$ & & $\cap$ & & & & $\cap$&    &$\cap$\\
 $\Lambda_k$ & $\subset$ & $W^{k,1}$ & $\subset$ & ...&$\subset$&$W^{k,k-s}$ &$\subset$ & $F_{b_1}^\perp$   
\end{tabular}
\end{center}
such that each row and each column forms a partial flag of isotropic subspaces. Here we omit $W^{\tau,j}$ if $j>z_\tau$. If $s=0$, then the rightmost column starts with $F_{b_k}^\perp$.
\end{Rem}

\begin{Thm}\label{smoothness}
The forgetful map $\pi:\Gamma_{a;b}(F_\bullet)\rightarrow \Sigma_{a;b}(F_\bullet)$ is a resolution of singularities.
\end{Thm}
\begin{proof}
    We prove that $\Gamma_{a;b}(F_\bullet)$ is smooth by showing that it can be realized as an iterated tower of Grassmannians and symplectic Grassmannians. Let $\pi_\tau$ denote the projection obtained by forgetting $\Lambda_\tau$, and let $\pi^{\tau,j}$ denote the projection obtained by forgetting $W^{\tau,j}$. Consider the fibration
    $$\pi^{1,k-s}\circ\pi^{1,k-s-1}\circ...\circ\pi^{1,1}\circ\pi_{1}\circ\pi^{2,k-s}\circ\pi^{2,k-s-1}\circ...\circ\pi^{k,1}\circ\pi_k.$$
    Here, the source of each projection is understood accordingly, and we omit $\pi^{\tau,j}$ whenever $W^{\tau,j}$ is undefined. The leftmost projection (whose source has only one component) is regarded as the constant map. We claim that all fibers are isomorphic to Grassmannians or symplectic Grassmannians.
    
\begin{itemize}

    \item (I-1) If $z_\tau=0$, then necessarily $\tau\leq s$, and the fibers of $\pi_\tau$ are isomorphic to $G(1,a_\tau-\tau+1)$ parameterizing all $\Lambda_\tau$ such that $$\Lambda_{\tau-1}\subset \Lambda_\tau\subset F_{a_\tau}.$$

        \item (I-2) If $z_\tau>0$, then the fibers of $\pi_\tau$ are isomorphic to $G(1,\dim(W^{\tau,1})-\tau+1)$, parameterizing all $\Lambda_\tau$ such that $$\Lambda_{\tau-1}\subset \Lambda_\tau\subset W^{\tau,1}.$$

                \item (II-1) If $1\leq j<z_\tau$ and $W^{\tau-1,j}$ does not exist, then the fibers of $\pi^{\tau,j}$ are isomorphic to the Grassmannian parameterizing all $W^{\tau,j}$ such that $$F_{b_j}\subset W^{\tau,j}\subset W^{\tau,j+1}$$

                    \item (II-2) If $1\leq j<z_\tau$ and $W^{\tau-1,j}$ exists, then the fibers of $\pi^{\tau,j}$ are isomorphic to the Grassmannian parameterizing all $W^{\tau,j}$ such that $$W^{\tau-1,j}\subset W^{\tau,j}\subset W^{\tau,j+1}$$
                    
                \item (III-1) If $\tau\leq s$ and $W^{\tau-1,z_\tau}$ does not exist, then the fibers of $\pi^{\tau,z_\tau}$ are isomorphic to the Grassmannian parameterizing all $W^{\tau,z_\tau}$ such that $$F_{b_{z_\tau}}\subset W^{\tau,z_\tau}\subset F_{a_\tau}$$
 
                \item (III-2) If $\tau\leq s$ and $W^{\tau-1,z_\tau}$ exists, then the fibers of $\pi^{\tau,z_\tau}$ are isomorphic to the Grassmannian parameterizing all $W^{\tau,z_\tau}$ such that $$W^{\tau-1,z_\tau}\subset W^{\tau,z_\tau}\subset F_{a_\tau}$$
                
                \item (III-3) If $\tau\geq s+1$, then $z_\tau=k-\tau+1$, and the fibers of $\pi^{\tau,z_\tau}$ are isomorphic to the symplectic Grassmannian parameterizing all $W^{\tau,z_\tau}$ such that $$W^{\tau-1,z_\tau}\subset W^{\tau,z_\tau}\subset F_{b_{k-\tau+1}}^\perp$$ 
    \end{itemize}
    We conclude that $\Gamma_{a;b}(F_\bullet)$ is smooth. Note that for a general point $\Lambda\in \Sigma_{a;b}(F_\bullet)$, the fiber of $\pi$ is uniquely determined by the relations $\Lambda_\tau=\Lambda\cap F_{a_\tau}$ or $\Lambda_\tau=\Lambda\cap F_{b_{k-\tau+1}}^\perp$, and $W^{\tau,j}=\text{span}(\Lambda_\tau,F_{b_j})$. By Zariski's Main Theorem, $\pi$ is an isomorphism over $\Sigma_{a;b}^\circ(F_\bullet)$. Therefore, $\pi:\Gamma_{a;b}(F_\bullet)\rightarrow \Sigma_{a;b}(F_\bullet)$ is a resolution of singularities.
\end{proof}

Let $e_1,...,e_n$ be a basis of $V$ such that $F_i=\text{span}(e_1,...,e_i)$ for $1\leq i\leq \frac{n}{2}$ and $F_j^\perp=\text{span}(e_1,...,e_{n-j})$ for $0\leq j\leq \frac{n}{2}-1$. Let $G=Sp(n)\cong Sp(V)$. Let $T$ and $B$ be the subgroups of diagonal matrices and upper triangular matrices in $G$, respectively. Note that the variety $\Gamma_{a;b}(F_\bullet)$ constructed above admits a left action of $B$ by acting on each component, and the resolution of singularities $\pi:\Gamma_{a;b}(F_\bullet)\rightarrow \Sigma_{a;b}(F_\bullet)$ is $B$-equivariant. We will prove that $\pi$ is a finite log resolution under certain conditions.
\begin{Ex}
    Consider the Schubert variety $\Sigma_{3;1}(F_\bullet)$ in $SG(2,n)$.
    Then $\Gamma_{3;1}(F_\bullet)\subset SG(1,n)\times SG(2,n)\times SG(2,n)\times SG(3,n)$ is the locus of $(\Lambda_1,\Lambda_2,W^1,W^2)$ such that $\Lambda_2\in\Sigma_{3;1}$ and fitting into the following diagram:
\begin{center}
\begin{tabular}{ c c c c c }
 &  & $F_1$ & & \\ 
  & & $\cap$ & & \\  
 $\Lambda_1$ & $\subset$ & $W^1$ & $\subset$ &$F_3$    \\
  $\cap$ & & $\cap$ &  &$\cap$\\
   $\Lambda_2$ & $\subset$ & $W^2$ & $\subset$ &$F_1^\perp$  
\end{tabular}
\end{center}
There is a natural projection 
$$\pi':\Gamma_{3;1}(F_\bullet)\rightarrow \Sigma_{3^1;1^2}(F_\bullet)\subset SF(1,2;n).$$
Note that the partial flag variety $SF(1,2;n)$ has a stratification by Schubert cells and hence has only finitely many $B$-orbits. We show that $\Gamma_{3;1}(F_\bullet)$ also has finitely many $B$-orbits by proving that, over every point $(\Lambda_1,\Lambda_2)\in \Sigma_{3^1;1^2}(F_\bullet)$, the fiber has finitely many orbits under the action of the stabilizer of $(\Lambda_1,\Lambda_2)$.

If $(\Lambda_1,\Lambda_2)$ is not contained in the exceptional locus of $\pi'$, then the fiber over it consists of a unique point. Thus the corresponding orbit lifts to a unique orbit in $\Gamma_{3;1}(F_\bullet)$.

If $(\Lambda_1,\Lambda_2)$ is contained in the exceptional locus of $\pi'$, then necessarily either $\Lambda_1=F_1$ or $F_1\subset \Lambda_2\subset F_3$. First assume that $\Lambda_1=F_1$ and $\Lambda_2\not\subset F_3$. The fiber over $(\Lambda_1,\Lambda_2)$ can then be identified with the locus of $W^1$ such that
$$F_1\subset W^1\subset F_3.$$
Projecting from $F_1$, we see that the stabilizer of $\Lambda_2/F_1$ contains the automorphism group of $F_3/F_1$, since $\Lambda_2\not\subset F_3$. Thus $G(1,F_3/F_1)$ is stratified into two Schubert cells under the action of the stabilizer.

If $\Lambda_1\neq F_1$ and $F_1\subset \Lambda_2\subset F_3$, then $W^1$ is uniquely determined by $\text{span}(\Lambda_1,F_1)$, and $W^2$ can be any $3$-dimensional isotropic subspace such that
$$\Lambda_2\subset W^2\subset F_1^\perp.$$
By projecting from $\Lambda_2$, one sees that the fiber is isomorphic to $SG(1,F_1^\perp/\Lambda_2)$, which has a stratification by Schubert cells. 
\end{Ex}

We need the following lemmas.

\begin{Lem}\label{orbit}
        Let $G=SL(V)$. Fix a maximal torus $T$ and a Borel subgroup $B\supset T$. Let $\Lambda_1\subset \cdots \subset \Lambda_t$ be a partial flag in $V$ with $\dim(\Lambda_i)=i$. Let $S\subset B$ be the stabilizer of $\Lambda_\bullet$ under the left action. Then the orbit space $G(1,V)/S$ is finite.
    \end{Lem}
\begin{proof}
Take a basis $e_1,\ldots,e_n$ of $V$ such that $T$ is the group of diagonal matrices and $B$ is the group of upper triangular matrices. Let $F_a$ denote the span of $e_1,\ldots,e_a$. Assume that the orbit $B\Lambda_\bullet$ is the Schubert cell $\Sigma_w^\circ$ in $F(1,\ldots,t;n)$ and that $\Lambda_\bullet$ is the unique $T$-fixed point in it. Then $S$ can be identified with
$$S=\{b\in B|b_{ij}=0 \text{ for any }(i,j)\text{ such that }i<j,w^{-1}(i)>w^{-1}(j)\}.$$
Each $B$-orbit in $G(1,V)$ can be identified with a Schubert cell
$$\Sigma^\circ_a:=\{v\in G(1,V)|v\subset F_a,v\not\subset F_{a-1}\}$$
We claim that for each Schubert cell, the orbit space $\Sigma_a^\circ/S$ is finite. Each $v\in \Sigma_a^\circ$ can be represented by a column vector such that $v_a=1$ and $v_j=0$ for $j>a$. Let $\bar{I}$ be the set of all $i$ such that $i<a$ and $b_{ia}=0$ for every $b\in S$. For each subset $I\subset \bar{I}$, set
$$O_I:=\{v\in \Sigma^\circ_a|v_i=0\text{ for }i\in I\text{ and }v_i\neq 0\text{ for }i\in \bar{I}-I\}.$$
It is then straightforward to check that the action of $S$ on each $O_I$ is transitive. Thus $|\Sigma^\circ_a/S|\leq 2^{|\bar{I}|}<\infty$. We conclude that $G(1,V)/S$ is finite.
\end{proof}

\begin{Lem}\label{SNC}
    Consider the Schubert variety $\Sigma_{a^\alpha;b^\beta}(F_\bullet)$ in $SF(k_1,k_2,\ldots,k_h;n)$, where $a_i=i$ for $1\leq i\leq k_h-1$, the sequence $\alpha_i$ is non-decreasing except possibly at $\alpha_{k_h}$ or $\beta_1$, and either $s=k_h$ with $\alpha_{k_h}=1$, or $s=k_h-1=b_1$ with $\beta_1=1$. Let $\pi_t:\Sigma_{a^\alpha;b^\beta}(F_\bullet)\rightarrow SG(k_t,n)$, $1\leq t\leq h$, be the canonical projection sending $(\Lambda_1,\ldots,\Lambda_h)$ to $\Lambda_t$. Then $\pi_1^{-1}(\partial \pi_1(\Sigma_{a^\alpha;b^\beta}))$ is an SNC divisor in $\Sigma_{a^\alpha;b^\beta}(F_\bullet)$.
\end{Lem}
\begin{proof}
    We use induction on $h$. For $1\leq t\leq h-1$, let $\pi^t:\Sigma_{a^\alpha;b^\beta}(F_\bullet)\rightarrow SF(k_{t+1},...,k_h;n)$ be the forgetful map sending $(\Lambda_1,\ldots,\Lambda_h)$ to $(\Lambda_{t+1},\ldots,\Lambda_h)$. Note that the images of $\Sigma_{a^\alpha;b^\beta}(F_\bullet)$ under $\pi_t$ and $\pi^t$ are Schubert varieties in the corresponding symplectic Grassmannians and symplectic flag varieties. In particular, $\pi_h(\Sigma_{a^\alpha;b^\beta})$ is isomorphic to $G(1,a_{k_h}-k_h+1)$ if $s=k_h$, and to $SG(1,n-2k_h+2)$ if $s=k_h-1$. Hence $\partial \pi_h(\Sigma_{a^\alpha;b^\beta})$ is smooth and irreducible.
    
     For each $2\leq t\leq h$, let $\pi_{(t)}: \pi^{t-1}(\Sigma_{a^\alpha;b^\beta})\rightarrow SG(k_t,n)$ be the natural projection. Suppose by induction that $\pi_{(2)}^{-1}(\partial\pi_2(\Sigma_{a^\alpha;b^\beta}))$ is an SNC divisor in $\pi^1(\Sigma_{a^\alpha;b^\beta})$. Notice that
    $$\pi_1^{-1}(\partial\pi_1(\Sigma_{a^\alpha;b^\beta}))=\tilde{D}\cup (\pi^{1})^{-1}(\pi_{(2)}^{-1}(\partial\pi_2(\Sigma_{a^\alpha;b^\beta}))),$$
    where $\tilde{D}$ is a locally trivial projective bundle over $\pi^1(\Sigma_{a^\alpha;b^\beta})$ with fiber $\mathbb{P}(F_{k_2-1}/F_{k_1-1})$. Since the pullback of an SNC divisor under a smooth morphism is again SNC, it follows that $$(\pi^{1})^{-1}(\pi_{(2)}^{-1}(\partial\pi_2(\Sigma_{a^\alpha;b^\beta})))$$ is an SNC divisor in $\Sigma_{a^\alpha;b^\beta}$. Moreover, $\tilde{D}$ is smooth, and it meets the pullback divisor transversely. Therefore $\pi_1^{-1}(\partial\pi_1(\Sigma_{a^\alpha;b^\beta}))$ is an SNC divisor in $\Sigma_{a^\alpha;b^\beta}$.
\end{proof}

\begin{Thm}\label{logsg}
  Assume $s=|a|=k-1$. Then $\pi:\Gamma_{a;b}(F_\bullet)\rightarrow \Sigma_{a;b}(F_\bullet)$ is a finite log resolution.
\end{Thm}
\begin{proof}
By Theorem \ref{smoothness}, it suffices to show that $\Gamma_{a;b}(F_\bullet)$ has finitely many $B$-orbits and that $\pi^{-1}(\partial \Sigma_{a;b}(F_\bullet))$ is an SNC divisor in $\Gamma_{a;b}(F_\bullet)$.

We first show that $\Gamma_{a;b}(F_\bullet)$ has finitely many $B$-orbits. Set $W^{0,1}=F_{b_1}$. We use induction on $\tau$. Let $\Gamma^{\tau,1}$ be the image of $\Gamma_{a;b}(F_\bullet)$ under the projection sending $(\Lambda_1,...,\Lambda_k,W^{1,1},W^{2,1}...,W^{k,1})$ to $(\Lambda_1,...,\Lambda_k,W^{1,1},W^{2,1},...,W^{\tau,1})$.

    Note that $\Gamma^{0,0}$ is the Schubert variety defined by the same sequence with the upper index $(1,2,\ldots,k)$ in the symplectic partial flag variety $SF(1,\ldots,k;n)$, and hence has finitely many $B$-orbits by the Bruhat decomposition.

    If $\tau=1$, consider the natural projection
    \begin{eqnarray}
    \Gamma^{1,1}&\longrightarrow&\Gamma^{0,0}\nonumber\\
        (\Lambda,...,\Lambda_k,W^{1,1})&\mapsto&(\Lambda_1,...,\Lambda_k).\nonumber
    \end{eqnarray}
    We claim that every $B$-orbit in $\Gamma^{0,0}$ lifts to finitely many $B$-orbits in $\Gamma^{1,1}$. Assume $a_1>b_1>0$; otherwise $W^{1,1}=0$, and the statement is trivial. Let $O$ be a $B$-orbit in $\Gamma^{0,0}$, and let $(\Lambda_1,\ldots,\Lambda_k)\in O$ be the unique $T$-fixed point. If $\Lambda_1\not\subset F_{b_1}$, then $W^{1,1}=\text{span}(\Lambda_1,F_{b_1})$ is uniquely determined, and the claim follows. If $\Lambda_1\subset F_{b_1}$, then $W^{1,1}$ can be any subspace of dimension $b_1+1$ such that
    $$F_{b_1}\subset W^{1,1}\subset F_{a_1}\cap \Lambda_k^\perp.$$
    Intersecting the partial flag $\Lambda_\bullet$ with $F_{a_1}$ and then quotienting by $F_{b_1}$, the claim follows from Lemma \ref{orbit}.

Now assume inductively that $|\Gamma^{\tau',\phi'}/B|$ is finite for all earlier stages. If $\Lambda_\tau\not\subset W^{\tau-1,\phi}$, then $W^{\tau,\phi}=\text{span}(\Lambda_\tau,W^{\tau-1,\phi})$ is uniquely determined. If $\Lambda_\tau\subset W^{\tau-1,\phi}$, then $W^{\tau,\phi}$ is parametrized by the Grassmannian $G(1,a_\tau-\dim(W^{\tau-1,\phi}))$ if $\tau\leq s$, or by $G(1,n-b_{k-\tau+1}-\dim(W^{\tau-1,\phi}))$ if $\tau>s$. By Lemma \ref{orbit}, each $B$-orbit in $\Gamma^{\tau-1,\phi}$ (or in $\Gamma^{k,\phi-1}$ if $\tau=1$) lifts to finitely many $B$-orbits in $\Gamma^{\tau,\phi}$. Therefore, by induction, $\Gamma_{a;b}(F_\bullet)$ has finitely many $B$-orbits.

 We next show that the complement of $\pi^{-1}(\Sigma_{a;b}^\circ(F_\bullet))$ in $\Gamma_{a;b}(F_\bullet)$ is an SNC divisor. We again use induction. First consider the case $\tau=1$, and assume first that $s\neq 0$. Consider the following Schubert variety in $SF(1,b_1+1,\ldots,b_{z_1}+1;n)$:
    $$\tilde{\Gamma}^1:=\{(\Lambda_1,W^{1,1},...,W^{1,z_1})|F_{b_j}\subset W^j\subset F_{a_1},  1\leq j\leq z_1\}.$$
    It is the image of $\Gamma_{a;b}(F_\bullet)$ under the projection forgetting $\Lambda_\tau$ and $W^{\tau,j}$ for $\tau>1$. Let $\pi_1:\tilde\Gamma^1\rightarrow SG(1,n)$ be the first projection. By Lemma \ref{SNC}, $\pi_1^{-1}(\partial\pi_1(\tilde{\Gamma}^1))$ is an SNC divisor in $\tilde{\Gamma}^1$. If $s=0$, then replacing $F_{a_1}$ by $F_{b_{k-s}}^\perp$, the same argument applies.

    Now let $\pi'$ be the projection from $\Gamma_{a;b}(F_\bullet)$ obtained by forgetting $\Lambda_k$ and $W^{k,j}$ for $1\leq j\leq z_k$. By induction, we may assume that $D:=(\pi')^{-1}(\pi'(\Gamma_{a;b}(F_\bullet))\backslash\pi'(\Gamma_{a;b}(F_\bullet))^\circ)$ is an SNC divisor. Then
    $$\pi^{-1}(\Sigma_{a;b}(F_\bullet)\backslash\Sigma_{a;b}^\circ(F_\bullet))=D\cup D',$$
where $D'$ is the divisor in $\Gamma_{a;b}(F_\bullet)$ such that $D'\cap(\pi')^{-1}(\Lambda_1,...,\Lambda_{k-1},W^{1,1},...,W^{k-1,z_{k-1}})$ is the Schubert variety
$$\{(\Lambda_k,W^{k,1},...,W^{k,z_k})|\Lambda_{k-1}\subset\Lambda_k\subset W^{k,1},W^{k-1,j}\subset W^{k,j}\subset W^{k,j+1}\}.$$
By Lemma \ref{SNC}, the divisor $D'$ restricts to an SNC divisor on each fiber of $\pi'$. Therefore $D\cup D'$ is an SNC divisor in $\Gamma_{a;b}(F_\bullet)$.
\end{proof}
\begin{Cor}
Assume $s=|a|=k-1$. Then each homogeneous component of $c_{SM}(\Sigma_{a;b}^\circ(F_\bullet))$ is Hodge rigid.
\end{Cor}
\begin{proof}
    Follows from Theorem \ref{logsg} and Theorem \ref{hodgerigid}.
\end{proof}

\subsection{Singular loci of Schubert varieties}\label{singularloci}

Another application of our construction of resolutions is the computation of the singular loci of Schubert varieties.

Let $E$ be the exceptional locus of $\pi:\Gamma_{a;b}(F_\bullet)\rightarrow \Sigma_{a;b}(F_\bullet)$.
\begin{Lem}\label{component}
    Assume that $Z$ is a component of $E$ of codimension at least $2$ in $\Gamma_{a;b}(F_\bullet)$. Then $\pi(Z)$ is contained in the singular locus of $\Sigma_{a;b}(F_\bullet)$.
\end{Lem}
\begin{proof}
    Let $p$ be a general point of $\pi(Z)$. Assume, for a contradiction, that $p$ is a smooth point of $\Sigma_{a;b}(F_\bullet)$. Let $U$ be an open neighborhood of $p$ such that $U$ is smooth and $\pi^{-1}(U)\cap Z=\pi^{-1}(U)\cap E$. Since $\pi^{-1}(U)$ is smooth, every Weil divisor on $\pi^{-1}(U)$ is Cartier. By assumption, $\pi^{-1}(U)\cap Z$ has codimension at least $2$ in $\pi^{-1}(U)$. Hence the pushforward of divisors induces an isomorphism $\pi_*: \text{Div}(\pi^{-1}(U))\cong\text{Div}(U)$. On the other hand, if $\pi^{-1}(U)\to U$ contracts a divisor only in codimension at least $2$, then a Cartier divisor on $\pi^{-1}(U)$ that has nontrivial degree on an exceptional curve cannot descend to $U$. Thus not every Cartier divisor on $\pi^{-1}(U)$ can come from $U$, contradicting the above isomorphism. Therefore $U$ cannot be smooth. This contradiction shows that $\pi(Z)$ is contained in the singular locus of $\Sigma_{a;b}(F_\bullet)$.
\end{proof}

From the perspective of Lemma \ref{component}, we first determine all components of $E$ of codimension at least $2$. Note that the image of the exceptional locus of $\pi$ is a union of subvarieties of the following types:
\begin{itemize}
    \item (I): The Zariski closure of $\{\Lambda\in \Sigma_{a;b}(F_\bullet)|\dim(\Lambda\cap F_{b_j})=x_j+1\}$, which we denote by $\Sigma(F_{b_j})$;
    \item (II): The Zariski closure of $\{\Lambda\in \Sigma_{a;b}(F_\bullet)|\dim(\Lambda\cap F_{a_i})=i+1\}$, which we denote by $\Sigma(F_{a_i})$;
    \item (III): The Zariski closure of $\{\Lambda\in \Sigma_{a;b}(F_\bullet)|\dim(\Lambda\cap F_{b_j}^\perp)=k-j+2\}$, which we denote by $\Sigma(F_{b_j}^\perp)$;
\end{itemize}

\begin{Thm}
Let $C$ be an irreducible component of $E$. If $\pi(C)$ is a subvariety of one of the following types, then $C$ has codimension at least $2$, and therefore $\pi(C)$ is contained in the singular locus of $\Sigma_{a;b}(F_\bullet)$:
\begin{itemize}
\item (I) $\Sigma(F_{b_j})$ when $b_j<a_s$;
\item (II) $\Sigma(F_{a_i})$, unless $i=s$, $a_s=\frac{n}{2}$, and $x_{k-s}=s+1-\frac{n-2b_{k-s}}{2}$;
\item (III) $\Sigma(F_{b_j}^\perp)$ for all $2\leq j\leq k-s$.
\end{itemize}
\end{Thm}
\begin{proof}
\textbf{(I-1-1)} Assume $b_j<a_s$, $b_j\neq a_i$ for all $i$, and $b_j\neq b_{j-1}-1$. Set $\tau_0:=x_j+1$. Further assume that there is no $j'>j$ such that $b_{j'}<a_{\tau_0}$; otherwise $\Sigma(F_{b_j})\subset \Sigma(F_{b_{j'}})$. Under this assumption, we must have $W^{\tau_0-1,j}=W^{\tau_0,j+1}=0$. The subvariety $\Sigma(F_{b_j})$ is the Schubert variety whose defining sequence is obtained by replacing $F_{a_{\tau_0}}$ with $F_{b_j}$ and replacing $F_{b_j}^\perp$ with $F_{a_{\tau_0}}^\perp$. Over a general point $\Lambda\in \Sigma(F_{b_j})$, the spaces $\Lambda_\tau$, $1\leq \tau\leq k$, are uniquely determined, and the spaces $W^{\tau,t}$ are uniquely determined unless $t=j$. Notice that $W^{\tau_0,j}$ can be chosen to be any subspace of dimension $b_j+1$ such that
$$\text{span}(\Lambda_{\tau_0},F_{b_j})\subset W^{\tau_0,j}\subset F_{a_{\tau_0}}.$$
Once $W^{\tau_0,j}$ is chosen, all the other $W^{\tau,j}$ are uniquely determined by the relation $W^{\tau,j}=\text{span}(W^{\tau_0,j},\Lambda_\tau)$. 
Thus
\begin{eqnarray}
    \dim(\pi^{-1}(\Lambda))&=&\dim(G(1,a_{\tau_0}-\dim(F_{b_j})))\nonumber\\
    &=&a_{\tau_0}-b_j-1.\nonumber
\end{eqnarray}

 \begin{eqnarray}
        \text{codim}(\pi^{-1}(\Sigma(F_{b_j})))&=&\dim(\Sigma_{a;b})-\dim(\Sigma(F_{b_j}))-\dim(\pi^{-1}(\Lambda))\nonumber\\
        &=&2a_{\tau_0}-2b_j-1-\dim(\pi^{-1}(\Lambda))\nonumber\\
        &=&a_{\tau_0}-b_j\geq 2\nonumber
        \end{eqnarray}
        
\textbf{(I-1-2)} Assume $b_j<a_s$, $b_j\neq a_i$ for all $i$, and $b_j=b_{j-1}+1$. Then the defining sequence deforms to one of smaller dimension, and hence the dimension of $\Sigma(F_{b_j})$ drops. On the other hand, the dimension of the general fiber does not change. We conclude that
$$\text{codim}(\pi^{-1}(\Sigma(F_{b_j})))\geq 2$$

 {\bf (I-2-1)} Assume $b_{j}=a_{\tau_0}$ for some $\tau_0<s$. Let $i_0$ be the smallest number such that $a_{\tau_0}-a_{i_0}=\tau_0-i_0$. Further assume that $b_{j'}\neq a_{i_0}-2$ for all $1\leq j'\leq k-s$. Note that the variety $\Sigma(F_{b_j})$ exists only if $a_{\tau_0}>\tau_0$. Let $F_{a_{i_0}-1}$ be an isotropic subspace of dimension $a_{i_0}-1$ contained in $F_{a_{i_0}}$ and containing both $F_{a_{i_0-1}}$ and $F_{b_{j-1}}$. Then $\Sigma(F_{b_j})$ is the Schubert variety whose defining sequence is obtained by replacing $F_{a_{\tau_0+1}}$ with $F_{a_{i_0}-1}$ and replacing $F_{b_j}^\perp$ with $F_{a_{\tau_0+1}}^\perp$. Over a general point $\Lambda$, the spaces $\Lambda_\tau$ are uniquely determined unless $\tau=\tau_0$, and the spaces $W^{\tau,t}$ are uniquely determined unless $\tau=\tau_0$ or $t=j$. The constraints on $\Lambda_{\tau_0}$ and $W^{\tau_0+1,j}$ are
$$\Lambda_{\tau_0-1}\subset \Lambda_{\tau_0}\subset \Lambda\cap F_{a_{\tau_0}}$$
and
$$F_{a_{\tau_0}}\subset W^{\tau_0+1,j}\subset F_{a_{\tau_0+1}}.$$
Once $\Lambda_{\tau_0}$ and $W^{\tau_0+1,j}$ are chosen, the remaining spaces $W^{\tau_0,t}$ and $W^{\tau,t}$ are uniquely determined. Thus the fiber dimension over $\Lambda$ is
$$\dim(\pi^{-1}(\Lambda))=a_{\tau_0+1}-a_{\tau_0}.$$
Therefore  \begin{eqnarray}
        \text{codim}(\pi^{-1}(\Sigma(F_{b_j})))&=&\dim(\Sigma_{a;b})-\dim(\Sigma(F_{b_j}))-\dim(\pi^{-1}(\Lambda))\nonumber\\
        &=&\tau_0-i_0+2+2a_{\tau_0+1}-2a_{\tau_0}-\dim(\pi^{-1}(\Lambda))\nonumber\\
        &=&\tau_0-i_0+2+a_{\tau_0+1}-a_{\tau_0}>1\nonumber
        \end{eqnarray}
        
\textbf{(I-2-2)} Assume $b_j=a_{\tau_0}$ for some $\tau_0<s$ and $b_{j-1}=a_{i_0}-2$ for some $i_0\leq \tau_0$ such that $a_{\tau_0}-a_{i_0}=\tau_0-i_0$. Then the defining sequence deforms to one of smaller dimension, and hence the dimension of $\Sigma(F_{b_j})$ drops. On the other hand, the dimension of the general fiber does not change. We conclude that
$$\text{codim}(\pi^{-1}(\Sigma(F_{b_j})))>1.$$

Next consider the subvarieties of type (II). If $a_i=i$ or $i=k$, then $\Sigma(F_{a_i})$ is empty. From now on assume $i<k$ and $a_i>i$. Note that if $a_i-a_{i'}=i-i'$, then $\Sigma(F_{a_i})=\Sigma(F_{a_{i'}})$. Therefore we may assume that $a_i$ is essential, that is, $a_i\neq a_{i+1}-1$. Let $i_0\leq i$ be the smallest integer such that $a_i-a_{i_0}=i-i_0$. 

\textbf{(II-1)} Assume $i<s$ and $a_i-b_j\neq 2$ for all $1\leq j\leq k-s$. If there exists $1\leq j\leq k-s$ such that $a_i\leq b_j<a_{i+1}$, then $\Sigma(F_{a_i})\subset \Sigma(F_{b_j})$. Now assume that no such $j$ exists. Then $\Sigma(F_{a_i})$ is the Schubert variety whose defining sequence is obtained by replacing $F_{a_{i+1}}$ with a subspace $F_{a_{i_0}-1}$ of dimension $a_{i_0}-1$ contained in $F_{a_{i_0}}$. Note that
    $$\dim(\Sigma_{a;b})-\dim(\Sigma({F_{a_i}}))=a_{i+1}-(a_{i_0}-1).$$ 
    Over a general point $\Lambda\in\Sigma(F_{a_i})$, the spaces $W^{\tau,t}$ are uniquely determined by $W^{\tau,t}=\text{span}(\Lambda_\tau,F_{b_t})$ for all $(\tau,t)$, and the spaces $\Lambda_\tau$ are uniquely determined by $\Lambda_\tau=\Lambda\cap F_{a_\tau}$ or $\Lambda_\tau=\Lambda\cap F_{b_{k-\tau+1}}^\perp$ for $1\leq \tau\leq i_0-1$ or $\tau\geq i+1$. For $\tau=i$, the space $\Lambda_i$ can be any subspace of dimension $i$ such that $\Lambda_{{i_0-1}}\subset\Lambda_i\subset \Lambda_{i+1}$. For a fixed general $\Lambda_i$, $\Lambda_{\tau}$ are uniquely determined by $\Lambda_{\tau}=\Lambda_i\cap F_{a_\tau}$ for $i_0\leq \tau\leq i-1$. Thus
        $$\dim(\pi^{-1}(\Lambda))=\dim(G(i-i_0+1,i-i_0+2))=i-i_0+1$$
        and
        \begin{eqnarray}
        \text{codim}(\pi^{-1}(\Sigma(F_{a_i})))&=&\dim(\Sigma_{a;b})-\dim(\Sigma(F_{a_i}))-\dim(\pi^{-1}(\Lambda))\nonumber\\
        &=&a_{i+1}-a_{i_0}-(i-i_0)\nonumber\\
        &>&a_i+1-a_{i_0}-(i-i_0)=1\nonumber
        \end{eqnarray}
        
 \textbf{(II-2)} Assume $i=s<k$ and $a_{i_0}-b_j\neq 2$ for all $1\leq j\leq k-s$. If $a_s\leq b_{k-s}$, then $\Sigma(F_{a_s})\subset \Sigma(F_{b_{k-s}})$. Now assume that $a_s>b_{k-s}$. The defining sequence of $\Sigma(F_{a_s})$ is obtained by replacing $F_{b_{k-s}}^\perp$ with a subspace $F_{a_{i_0}-1}$ of dimension $a_{i_0}-1$ contained in $F_{a_{i_0}}$. Then
    $$\dim(\Sigma_{a;b})-\dim(\Sigma_{F_{a_s}})=n-b_{k-s}+x_{k-s}-a_{i_0}-s+1.$$
    Over a general point $\Lambda\in \Sigma(F_{a_s})$, an argument similar to that in \textbf{(II-1)} shows that the fiber has dimension $s-i_0+1$. Thus
        \begin{eqnarray}
        \text{codim}(\pi^{-1}(\Sigma(F_{a_i})))&=&\dim(\Sigma_{a;b})-\dim(\Sigma(F_{a_i}))-\dim(\pi^{-1}(\Lambda))\nonumber\\
        &=&n-b_{k-s}+x_{k-s}-a_{i_0}-s+1-(s-i_0+1)\nonumber\\
        &=&n-b_{k-s}+x_{k-s}-2s-(a_s-s)\nonumber\\
        &\geq&n-b_{k-s}-a_s-s+s+1-\frac{n-2b_{k-s}}{2}\nonumber\\
        &=&\frac{n}{2}-a_s+1\geq1\nonumber
        \end{eqnarray}
        The equality holds if and only if $a_s=\frac{n}{2}$ and $x_{k-s}= s+1-\frac{n-2b_{k-s}}{2}$. 

{\bf (III)} Assume $j\geq 2$. Let $j_0\geq j$ be the largest integer such that $b_{j_0}-b_j=j_0-j$. The subvariety $\Sigma(F_{b_j}^\perp)$ is the Schubert variety whose defining sequence is obtained by replacing $F_{b_{j-1}}^\perp$ with $F_{b_{j_0}+1}^\perp$. Then
    $$\dim(\Sigma_{a;b})-\dim(\Sigma(F^\perp_{b_j}))=j_0+b_{j}-b_{j-1}+x_{j-1}-x_j.$$ 
    Over a general point $\Lambda\in \Sigma(F_{b_j}^\perp)$, the spaces $\Lambda_\tau$ are uniquely determined unless $\tau=k-j+1$, and the spaces $W^{\tau,t}$ are all uniquely determined. The linear space $\Lambda_{k-j+1}$ can be any subspace of dimension $k-j+1$ such that $\Lambda_{{k-j_0}}\subset\Lambda_{k-j+1}\subset \Lambda_{k-j+2}$. Thus
        $$\dim(\pi^{-1}(\Lambda))=\dim(G(j_0-j+1,j_0-j+2))=j_0-j+1$$
        and
        \begin{eqnarray}
        \text{codim}(\pi^{-1}(\Sigma(F^\perp_{b_j})))&=&\dim(\Sigma_{a;b})-\dim(\Sigma(F^\perp_{b_j}))-\dim(\pi^{-1}(\Lambda))\nonumber\\
        &=&b_{j}-b_{j-1}+x_{j-1}-x_j+j-1>1\nonumber\nonumber
        \end{eqnarray}
\end{proof}

The following cases remain inconclusive:
\begin{itemize}
    \item $\Sigma(F_{b_j})$ when $b_j\geq a_s$;
    \item $\Sigma(F_{a_s})$ when $a_s=\frac{n}{2}$ and $x_{k-s}= s+1-\frac{n-2b_{k-s}}{2}$;
\end{itemize}

For these remaining cases, we compute the dimension of the tangent space at each point. Let $\Lambda\in \Sigma_{a;b}$, and let $\gamma(t)$ be a smooth path in $\Sigma_{a;b}$ with $\gamma(0)=\Lambda$. Choose a basis $e_1,\ldots,e_k$ of $\Lambda$ such that $e_\tau\in F_{a_\tau}$ for $1\leq \tau\leq s$ and $e_\tau\in F_{b_{k-\tau+1}}^\perp$ for $s+1\leq \tau\leq k$. Identify $\Lambda$ with $e_1\wedge e_2...\wedge e_k$, and identify $\gamma(t)$ with $e_1(t)\wedge...\wedge e_k(t)$, where each $e_\tau(t)$ is a smooth path in $F_{a_\tau}$ or $F_{b_{k-\tau+1}}^\perp$, respectively. Then the tangent vector $\gamma'(0)$ can be identified with
$$e_1'(t)\wedge e_2(t)\wedge...\wedge e_k(t)+...+e_1(t)\wedge...\wedge e_k'(t)\ \ \  (\text{mod}\  \Lambda)$$
Notice that for $1\leq \tau\leq s$, the vector $e_\tau'(0)$ belongs to $T_{e_\tau}F_{a_\tau}/\Lambda_\tau$, while for $s+1\leq \tau\leq k$, the vector $e_\tau'(0)$ belongs to $T_{e_\tau}(Q\cap F_{b_{k-\tau+1}}^\perp)\cap \Lambda_\tau^\perp/\Lambda_\tau$. We have
$$\dim(T_{e_\mu}F_{a_\mu}/\Lambda_\mu)=a_\mu-\mu,$$
and
$$\dim(T_{e_\tau}(Q\cap F_{b_{k-\tau+1}}^\perp)\cap \Lambda_\tau^\perp/\Lambda_\tau)=n-2\tau+x_{k-\tau+1}-\dim(\Lambda_\tau\cap F_{b_{k-\tau+1}}).$$
We now consider the remaining cases.
\begin{itemize}
    \item $\Sigma(F_{b_j})$ when $b_j\geq a_s$. In this case, the only term that changes is $$\dim(T_{e_{k-j+1}}(Q\cap F_{b_j}^\perp)\cap \Lambda_{k-j+1}^\perp/\Lambda_{k-j+1}).$$
    Its dimension increases by $1$, since $\Lambda_{k-j+1}$ intersects $F_{b_j}$ in one extra dimension. Thus the tangent space is larger than expected, and therefore $\Sigma(F_{b_j})$ is contained in the singular locus.
\item $\Sigma(F_{a_s})$ when $a_s=\frac{n}{2}$ and $x_{k-s}= s+1-\frac{n-2b_{k-s}}{2}$; In this case, $e_{s+1}$ now belongs to $F_{a_s}$. Thus the difference is
\begin{eqnarray}
    n-b_{k-s}-2(s+1)+x_{k-s}-(a_{k-s}-(s+1))&=&0\nonumber
\end{eqnarray}
Therefore $\Sigma(F_{a_s})$ is contained in the smooth locus.
\end{itemize}

Thus we obtain the following description of singular locus of Schubert varieties.
\begin{Cor}
The singular locus of $\Sigma_{a;b}$ is the union of the following subvarieties:
\begin{itemize}
    \item $\Sigma(F_{b_j})$ for all $1\leq j\leq k-s$;
    \item $\Sigma(F_{a_i})$ unless $i=s$ and $a_s=\frac{n}{2}$ and $x_{k-s}= s+1-\frac{n-2b_{k-s}}{2}$;
    \item $\Sigma(F_{b_j}^\perp)$ for all $2\leq j\leq k-s$.
\end{itemize}
\end{Cor}

\begin{Cor}
The smooth Schubert varieties in $SG(k,n)$ are sub-Grassmannians or sub-symplectic Grassmannians.    
\end{Cor}

\bibliographystyle{plain}

\end{document}